\documentclass[a4paper,10pt, oneside]{article}
\usepackage{authblk}
\usepackage{amsmath,amssymb,amsthm,mathrsfs,graphicx}
\usepackage{floatrow,float}
\usepackage[titletoc, title]{appendix}
\usepackage[colorlinks,linkcolor=blue,citecolor=blue]{hyperref}
\usepackage{etoolbox}
\usepackage{longtable}
\usepackage{diagbox}
\usepackage{booktabs,makecell,multirow}
\usepackage[capitalise,nosort]{cleveref}
\usepackage{cases,color}

\crefname{equation}{}{}
\crefname{lemma}{Lemma}{Lemmas}
\crefname{theorem}{Theorem}{Theorems}
\crefname{discr}{Discretization}{Discretizations}

\DeclareMathOperator{\D}{D}

\newcommand{\dual}[1]{\langle {#1} \rangle}
\newcommand{\Dual}[1]{\left\langle {#1} \right\rangle}

\newcommand{\nm}[1]{\lVert {#1} \rVert}
\newcommand{\bnm}[1]{\big\lVert {#1} \big\rVert}
\newcommand{\Nm}[1]{\left\lVert {#1} \right\rVert}
\newcommand{\snm}[1]{\lvert {#1} \rvert}

\newcommand{\Snm}[1]{\left\lvert {#1} \right\rvert}
\newcommand{\ssnm}[1]
{
  \left\vert\kern-0.25ex
  \left\vert\kern-0.25ex
  \left\vert
  {#1}
  \right\vert\kern-0.25ex
  \right\vert\kern-0.25ex
  \right\vert
}

\newtheorem{lemma}{Lemma}[section]
\newtheorem{remark}{Remark}[section]
\newtheorem{theorem}{Theorem}[section]

\AtBeginDocument{
	\label{CorrectFirstPageLabel}
	
}

\begin{document}
\title{ Numerical analysis of two Galerkin discretizations with graded temporal grids for fractional evolution equations}


\author[1]{Binjie Li\thanks{libinjie@scu.edu.cn, libinjie@aliyun.com}}
\author[2]{Tao Wang\thanks{Corresponding author: wangtao5233@hotmail.com}}
\author[1]{Xiaoping Xie\thanks{xpxie@scu.edu.cn}}
\affil[1]{School of Mathematics, Sichuan University}
\affil[2]{South China Reasearh Center for Applied Mathmatics and Interdisciplinary Studies, South China Normal University}
%



\maketitle

\begin{abstract}{
Two numerical methods with graded temporal grids are analyzed for  fractional evolution equations.	One is a low-order discontinuous Galerkin (DG) discretization in the case of fractional order $0<\alpha<1$, and the other one is a low-order Petrov Galerkin (PG) discretization in the case of fractional order $1<\alpha<2$. By a new duality technique, pointwise-in-time error estimates of  first-order and  $ (3-\alpha) $-order temporal accuracies are respectively derived for DG and PG, under  reasonable regularity assumptions on the initial value.  Numerical experiments are performed to verify the theoretical results.}
\end{abstract}

\medskip\noindent{\bf Keywords:} fractional diffusion-wave equation, graded temporal grid, convergence


\section{Introduction}
Let $ X $ be a separable Hilbert space with inner product $ (\cdot,\cdot)_X $.
Assume that the linear operator $ A: D(A) \subset X \to X $ is densely defined
and admits a bounded inverse $ A^{-1}: X \to X $, which is compact, symmetric
and positive. Consider the following time  fractional evolution equation:
\begin{equation}
\label{eq:model}
(\D_{0+}^\alpha (u-u_0))(t) + A u(t) = 0,
\quad 0 < t \leqslant T,
\end{equation}
where $ \alpha \in (0,2)\setminus \{1\} $, $ 0 < T < \infty $, $ u_0 \in X $ and
$ \D_{0+}^\alpha $ is a Riemann-Liouville   fractional derivative operator of
order $ \alpha $. Note that \cref{eq:model} is  usually called as a time fractional diffusion or wave equation when $A$ is a second order elliptic operator.



There are quite a few research works on the numerical
treatment of     time fractional evolution equations.  Let us briefly introduce four types of numerical methods for the discretization of  time fractional evolution equations.  
The first-type methods use convolution quadrature to approximate the fractional integral (derivative). These methods is very effective,  but they require the temporal grid to be uniform  (cf. \cite{Lubich1988,Lubich1996,Cuesta2006,Zeng2013,Jin2016}).  The second-type methods  use L1 scheme to approximate the fractional derivative (cf. \cite{Sun2006,Jin2016-L1,Yan2018,Li-Wang-Xie2019Wave,Liao2018}). Such methods are popular and easy to implement. 
  The third-type methods are spectral methods (cf. \cite{Li2018-wave, yang2016spectral,Li2009,Mao2016Efficient,ZayernouriA}), which use nonlocal basis functions to approximate the solution. The accuracy of spectral methods is high, provided that the solution or data is smooth enough. The fourth-type methods are  finite element methods (cf. \cite{Mustapha2012Uniform,Mustapha2013,Mclean2015Time,Luo2019,Li2018Aspace,Li2020}), which use local basis functions to approximate the solution. These methods are time-stepping, and easy to design   high order schemes.  It should be mentioned that the finite element method is identical to the L1 scheme in some cases (cf. \cite{Jin2017Discrete,Li-Wang-Xie2019Wave}).

Most of  the convergence analyses for   the  numerical  methods mentioned above are based on the assumption that   the exact solution is smooth enough. However,   the solution of a  fractional equation generally has singularity near the origin despite how smooth the data is (cf. \cite{Jin2016,Li2019-frac-diffu}).  In fact, the main difficulty is to derive the error estimates without any regularity restriction on the solution, especially for the case with nonsmooth data. When using uniform temporal grids, the Laplace transform technique is a  powerful tool for  error estimation in case of nonsmooth data (cf. \cite{Lubich1996,Cuesta2006,Mclean2015Time,Jin2016-L1,Yan2018,Li-Wang-Xie2019Wave}).   
We note that the non-uniform temporal grids is also useful to handle the singularity of fractional equations (cf. \cite{Mclean2009,Stynes2017,Liao2018,MustaphaAn}), but the corresponding numerical analysis seems rather complicated.

McLean and Mustapha \cite{Mclean2009} analyzed DG methods with graded temporal grids for    a variant form of \cref{eq:model}: 
\begin{equation}\label{eq:model2}
\begin{aligned}
\partial_{t} u + \D_{0+}^{1-\alpha} A u(t) &= 0, \quad 0<t\leqslant T, \\
u(0)&=u_0,
\end{aligned}
\end{equation}
which is obtained by applying $\D^{1-\alpha}_{0+}$ to the both sides of $\cref{eq:model}$. For \cref{eq:model2} with $0<\alpha<1$,   they derived first-order temporal accuracy for a piecewise-constant DG  under the condition that $u_0 \in D(A^{\nu})$ for $\nu>0$. For  the case $1<\alpha<2$,  they proved optimal error bounds for the piecewise-constant DG and a piecewise-linear DG  under the condition that 
\begin{align*}
t \nm{A \partial_{t}u(t)}_X+t^2\nm{A\partial_{tt}u(t)}_X &\leqslant Ct^{\sigma-1}, \quad  0<t\leqslant T,\\
 \nm{\partial_{t}u(t)}_X+t\nm{\partial_{tt}u(t)}_X &\leqslant Ct^{\sigma-1}, \quad    0<t\leqslant T,
\end{align*}
where $\sigma>0$ is a constant. For a  fractional reaction-subdiffusion equation,   Mustapha \cite{MustaphaAn}  derived second-order  temporal accuracy  for the L1 scheme with graded temporal grids    under the condition that 
	\[
	\nm{u(t)}_{H^2} \leqslant C, \quad  \nm{\partial_{t} u(t)}_{H^2}+t^{1-\alpha/2}\nm{\partial_{tt} u (t)}_{H^1} + t^{2-\alpha/2}\nm{\partial_{ttt}u(t)}_{H^1}\leqslant C t^{\sigma-1}, 
	\]
	for all $0<t\leqslant T$. 

Though being equivalent to \eqref{eq:model2} in some sense, equation \cref{eq:model} leads to different kinds of numerical methods.  For the fractional diffusion equation with nonsmooth data, Li et~al. \cite{Li2019SIAM} obtained  optimal error estimates for a low order DG.  It should be noticed that their analysis is optimal in the sense of some space-time Sobolev norms, which is not very sharp compare with the pointwise-in-time error estimates. For  a fractional diffusion equation, Stynes et~al. \cite{Stynes2017} analyzed  the L1 scheme with graded temporal
grids and derived temporal accuracy $ O(N^{\alpha-2}) $ ($ N $ is the number of nodes in
the temporal grid)  
 under the condition that
\[
\nm{\partial_x^{(4)}u(t)}_{L^\infty} \leqslant C, 
\quad 
\nm{\partial_{tt}u(t)}_{L^\infty} \leqslant C t^{\alpha-2}, \quad  0< t \leqslant T .
\]
 Liao et~al. \cite{Liao2018} obtained
temporal accuracy $ O(N^{\alpha-2}) $ for a reaction-subdiffusion equation by assuming that
\[
\nm{\partial_x^{(4)} u(t)}_{L^2} \leqslant C,
\quad 
\nm{\partial_{tt} u(t)}_{L^2} \leqslant C t^{\sigma-2},  \quad  0<t\leqslant T,
\]
where $ \sigma \in (0,2) \setminus \{1\} $.  
Although the regularity assumptions  above are reasonable in some situations, it is worthwhile to carry out  error estimation for some numerical methods with lesser regularity assumptions on the data. Moreover, as far as we know, there is no rigorous numerical analysis for \cref{eq:model} with $1<\alpha<2$ and graded temporal grids.


In this paper, we consider the DG and PG approximations for time fractional evolution equation \cref{eq:model} with $0<\alpha<1$  and $1<\alpha<2$ respectively. These methods are identical to the L1 scheme  when the temporal grid is uniform. We develop a new duality technique for the pointwise-in-time  error estimation, which is inspired by the local error estimation for the standard linear finite element method \cite{schatz1998pointwise,Brenner2008}. The key point of the analysis is the weighted estimate of a  ``regularized Green function" (cf.~\cref{lem:wg,lem:G2}).  
 For $0<\alpha<1$ and  $ u_0 \in D(A^\nu) $ with $ 0 < \nu \leqslant 1 $, we obtain the first-order temporal accuracy for the DG approximation with graded grids (cf.~\cref{thm:conv-diffu}).
For $1<\alpha<2$ and $ u_0 \in D(A^\nu) $ with $ 1/2 < \nu \leqslant 1 $, we obtain the $(3-\alpha)$-order temporal accuracy for the PG approximation with  graded grids (cf.~\cref{thm:conv-wave}).

The rest of this paper is organized as follows.  \cref{sec:pre} gives some notations and basic results, including Sobolev spaces, fractional calculus operators, spectral decomposition of $A$, solution theory and discretization spaces. \cref{sec:diffu} and   \cref{sec:wave}  establish  the error estimates for problem \cref{eq:model} with $0<\alpha<1$ and $1<\alpha<2$ respectively.  \cref{sec:numer} performs two numerical experiments to verify the theoretical results.
The last section is a conclusion.
\section{Preliminaries} 
\label{sec:pre}
Throughout this paper, we will use the following conventions: if $ \omega
\subset \mathbb R $ is an interval, then $ \dual{p,q}_\omega $ denotes the
Lebesgue or Bochner integral $ \int_\omega p q $ for scalar or vector valued
functions $ p $ and $ q $ whenever the integral makes sense; for a Banach space
$ W $, we use $ \dual{\cdot,\cdot}_W $ to denote a duality paring between $ W^*
$ (the dual space of $ W $) and $ W $; the notation $ C_\times $ denotes a
positive constant depending only on its subscript(s), and its value may differ
at each occurrence; for any function $ v $ defined on $ (0,T) $, by $ v(t-) $, $
0 < t \leqslant T $ we mean $ \lim_{s \to {t-}} v(s) $ whenever this limit
exists; given $ 0 < a \leqslant T $, the notation $ (a-t)_{+} $ denotes a
function of variable $ t $ defined by
\[
  (a-t)_{+} := \begin{cases}
    a-t & \text{ if } 0 \leqslant t < a, \\
    0 & \text{ if } a \leqslant t \leqslant T.
  \end{cases}
\]

\medskip\noindent{\bf Sobolev spaces.} Assume that $ -\infty < a < b < \infty $.
For any $ m \in \mathbb N $, define
\[
  {}_0H^m(a,b) := \left\{
    v \in H^m(a,b):\ v^{(k)}(a) = 0 \quad \forall 0 \leqslant k < m
  \right\}
\]
and endow this space with the norm
\[
  \nm{v}_{{}_0H^m(a,b)} := \nm{v^{(m)}}_{L^2(a,b)}
  \quad \forall v \in {}_0H^m(a,b),
\]
where $ H^m(a,b) $ is an usual Sobolev space and $ v^{(k)} $, $ 1 \leqslant k
\leqslant m $, is the $k$-th order weak derivative of $ v $. For any $ m \in
\mathbb N_{>0} $ and $ 0 < \theta < 1 $, define
\[
  {}_0H^{m-1+\theta}(a,b) :=
  ({}_0H^{m-1}(a,b), {}_0H^m(a,b))_{\theta,2},
\]
where $ (\cdot,\cdot)_{\theta,2} $ means the interpolation space defined by the
$ K $-method  \cite{Lunardi2018}. The space $ {}^0H^\gamma(a,b) $, $ 0
\leqslant \gamma < \infty $, is defined analogously. For each $ -\infty < \gamma
\leqslant 0 $, we use $ {}_0H^\gamma(a,b) $ and $ {}^0H^\gamma(a,b) $ to denote
the dual spaces of $ {}^0H^{-\gamma}(a,b) $ and $ {}_0H^{-\gamma}(a,b) $,
respectively. The embedding $ L^2(a,b) \hookrightarrow {}_0H^{-\gamma}(a,b) $, $
\gamma > 0 $, is understood in the conventional sense that
\[
  \dual{v,w}_{{}^0H^\gamma(a,b)} := \dual{v,w}_{(a,b)}
  \quad \forall w \in {}^0H^\gamma(a,b),
  \quad \forall v \in L^2(a,b).
\]
We will also use the following space:
\[
  H^{2\theta}(a,b) :=
  \left( L^2(a,b), H^2(a,b) \right)_{\theta,2},
  \quad \theta \in (0,1).
\]
Note that if $ 0 < \gamma < 1/2 $ then
\[
  {}_0H^\gamma(a,b) = {}^0H^\gamma(a,b) = H^\gamma(a,b)
  \quad\text{with equivalent norms.}
\]

\medskip\noindent{\bf Fractional calculus operators.} Assume that $ -\infty < a
< b < \infty $. For $ -\infty < \gamma < 0 $, define
\begin{align*}
  \big(\D_{a+}^\gamma v\big)(t) &:=
  \frac1{ \Gamma(-\gamma) }
  \int_a^t (t-s)^{-\gamma-1} v(s) \, \mathrm{d}s, \quad a < t < b, \\
  \big(\D_{b-}^\gamma v\big)(t) &:=
  \frac1{ \Gamma(-\gamma) }
  \int_t^b (s-t)^{-\gamma-1} v(s) \, \mathrm{d}s, \quad a < t < b,
\end{align*}
for all $ v \in L^1(a,b) $, where $ \Gamma(\cdot) $ is the gamma function. In
addition, let $ \D_{a+}^0 $ and $ \D_{b-}^0 $ be the identity operator on $
L^1(a,b) $. For $ j - 1 < \gamma \leqslant j $ with $ j \in \mathbb N_{>0} $,
define
\begin{align*}
  \D_{a+}^\gamma v & := \D^j \D_{a+}^{\gamma-j}v, \\
  \D_{b-}^\gamma v & := (-\D)^j \D_{b-}^{\gamma-j}v,
\end{align*}
for all $ v \in L^1(a,b) $, where $ \D $ is the first-order differential
operator in the distribution sense. The vector-valued version fractional
calculus operators are defined analogously. Assume that $ 0 < \beta \leqslant
\gamma < \beta + 1/2 $. For any $ v \in {}_0H^\beta(a,b) $, define $
\D_{a+}^\gamma v \in {}_0H^{\beta-\gamma}(a,b) $ by that
\[
  \Dual{\D_{a+}^\gamma v,w}_{{}^0H^{\gamma-\beta}(a,b)} :=
  \Dual{\D_{a+}^\beta v, \D_{b-}^{\gamma-\beta} w}_{(a,b)}
\]
for all $ w \in {}^0H^{\gamma-\beta}(a,b) $. For any $ v \in {}^0H^\beta(a,b) $,
define $ \D_{b-}^\gamma v \in {}^0H^{\beta-\gamma}(a,b) $ by that
\[
  \Dual{\D_{b-}^\gamma v,w}_{{}_0H^{\gamma-\beta}(a,b)} :=
  \Dual{\D_{b-}^\beta v, \D_{a+}^{\gamma-\beta} w}_{(a,b)}
\]
for all $ w \in {}_0H^{\gamma-\beta}(a,b) $. By \cref{lem:regu-basic} and a
standard density argument, it is easy to verify that the above definitions are
well-defined and that if
\[
  \Dual{\D_{a+}^\gamma v, w}_{{}^0H^{\beta_1}(a,b)}
  \quad \text{ and } \quad
  \Dual{\D_{a+}^\gamma v, w}_{{}^0H^{\beta_2}(a,b)}
\]
both make sense by the definition, then they are identical.

\medskip\noindent{\bf Spectral decomposition of $ A $.} Assume that the  separable Hilbert space $ X $ is
infinite dimensional. It is well known that (cf. \cite{Zeidler1995})
there exists an orthonormal basis, $\{\phi_n: n \in \mathbb N \} \subset D(A) $, 
of $ X $ such that
\[
  A \phi_n =\lambda_n \phi_n,
\]
where $ \{ \lambda_n: n \in \mathbb N \} $ is a positive non-decreasing sequence
and $\lambda_n\to\infty$ as $n\to\infty$. For any $ -\infty< \beta < \infty $,
define
\[
  D(A^{\beta/2}) := \left\{
    \sum_{n=0}^\infty c_n \phi_n:\
    \sum_{n=0}^\infty \lambda_n^\beta c_n^2 < \infty
  \right\}
\]
and equip this space with the norm
\[
  \Big\|\sum_{n=0}^\infty c_n \phi_n  \Big\|_{D(A^{\beta/2})}
  := \left(
    \sum_{n=0}^\infty \lambda_n^\beta c_n^2
  \right)^{1/2}.
\]

\medskip\noindent{\bf Solution theory.} For any $ \beta >0 $, define the
Mittag-Leffler function $ E_{\alpha,\beta}(z) $ by
\[
  E_{\alpha,\beta}(z) := \sum_{k=0}^\infty
  \frac{z^k}{\Gamma(k\alpha + \beta)}  \quad \forall z \in \mathbb C,
\]
which admits the following growth estimate (cf. \cite{Podlubny1998}):
\begin{equation}
  \label{eq:ml_grow}
  \snm{E_{\alpha,\beta}(-t)} \leqslant
  \frac{C_{\alpha,\beta}}{1+t}
  \quad \forall t > 0.
\end{equation}
For any $ \lambda > 0 $, a straightforward calculation yields
\begin{equation}
  \label{eq:ode}
  \D_{0+}^\alpha \left( E_{\alpha,1}(-\lambda t^\alpha) - 1 \right) +
  \lambda E_{\alpha,1}(-\lambda t^\alpha) = 0
  \quad \forall t \geqslant 0.
\end{equation}
Therefore, the solution to problem \cref{eq:model} is of the form
(cf. \cite{Sakamoto2011})
\begin{equation}
  \label{eq:u}
  u(t) = \sum_{n=0}^\infty E_{\alpha,1}(-\lambda_n t^\alpha)
  (u_0, \phi_n)_X \, \phi_n, \quad 0 \leqslant t \leqslant T.
\end{equation}
For any $ 0 < t \leqslant T $, a straightforward calculation gives
\begin{align*}
  & u'(t) = -\sum_{n=0}^\infty \lambda_n t^{\alpha-1}
  E_{\alpha,\alpha}(-\lambda_n t^\alpha)
  (u_0,\phi_n)_X \phi_n, \\
  & u''(t) = -\sum_{n=0}^\infty \lambda_n t^{\alpha-2}
  E_{\alpha,\alpha-1}(-\lambda_n t^\alpha)
  (u_0,\phi_n)_X \phi_n.
\end{align*}
Hence, for $1<\alpha<2$, by \cref{eq:ml_grow} we obtain that 
\begin{align}
  t^{-1} \nm{u'(t)}_X + \nm{u''(t)}_X
  & \leqslant C_\alpha t^{\alpha\nu-2} \nm{u_0}_{D(A^\nu)},
  \label{eq:u'-l2-growth} \\
  t^{-1} \nm{u'(t)}_{D(A^{1/2})} + \nm{u''(t)}_{D(A^{1/2})}
  & \leqslant C_\alpha t^{\alpha(\nu-1/2)-2}
  \nm{u_0}_{D(A^\nu)}, \label{eq:u'-h1-growth}
\end{align}
where $ 0 \leqslant \nu \leqslant 1 $.

\medskip\noindent{\bf Discretization spaces.} Let $ t_j := (j/J)^\sigma T $ for
each $ 0 \leqslant j \leqslant J $, where $ J \in \mathbb N_{>0} $ and $ \sigma
\geqslant 1 $. Define
\begin{align*}
  W_\tau &:= \big\{
    v \in L^\infty(0,T;D(A^{1/2})\!):
    \text{$ v $ is constant on $ (t_{j-1},t_j) $ for each $ 1 \leqslant j
    \leqslant J $}
  \big\}, \\
  W_\tau^\text{c} &:= \big\{
    v \in C([0,T];D(A^{1/2})):\,
    \text{$ v $ is linear on $ (t_{j-1},t_j) $ for each $ 1 \leqslant j
    \leqslant J $}
  \big\}.
\end{align*}
For the particular case $ D(A)=\mathbb R
$, we use $ \mathcal W_\tau $ and $ \mathcal W_\tau^\text{c} $ to denote $ W_\tau $
and $ W_\tau^\text{c} $, respectively. Assume that $ Y = X \text{ or } \mathbb R $. For any $ v \in L^1(0,T;Y) $ and
$ w \in C([0,T];Y) $, define $ Q_\tau v \in L^\infty(0,T;Y) $ and $ \mathcal
I_\tau w \in C([0,T];Y) $ respectively by 
\begin{small}
\begin{align*}
  (Q_\tau v)(t) :=
  \frac1{t_j-t_{j-1}} \int_{t_{j-1}}^{t_j} v
  \,\,\, \text{ and } \,\,\,
  (\mathcal I_\tau w)(t) :=
  \frac{t_j-t}{t_j-t_{j-1}} w(t_{j-1}) +
  \frac{t-t_{j-1}}{t_j-t_{j-1}} w(t_j)
\end{align*}
\end{small}
for all $ t_{j-1} < t < t_j $ and $ 1 \leqslant j \leqslant J $. In the sequel,
we will always assume that $ \sigma \geqslant 1 $.

\section{Fractional diffusion equation (\texorpdfstring{$ 0 < \alpha < 1 $}))}
\label{sec:diffu}
This section considers the following discretization: seek $ U \in W_\tau $ such
that
\begin{equation} 
  \label{eq:U-diffu}
  \int_0^T \big( (\D_{0+}^\alpha + A)U, V \big)_X \, \mathrm{d}t =
  \int_0^T \big( \D_{0+}^\alpha u_0, V \big)_X \, \mathrm{d}t \quad \forall  V \in W_\tau .
\end{equation}
\begin{remark}
  By \cref{eq:u}, a straightforward calculation yields that
  \begin{equation}
    \label{eq:u-weak-diffu}
    \int_0^T \big( (\D_{0+}^\alpha u + A)u, V \big)_X \, \mathrm{d}t =
    \int_0^T \big( \D_{0+}^\alpha u_0, V \big)_X \, \mathrm{d}t
    \quad \forall V \in W_\tau.
  \end{equation}
\end{remark}

\begin{remark} 
 We note that  when using uniform temporal grids, the  discretization \cref{eq:U-diffu} is equivalent to
the L1 scheme \cite{Jin2017Discrete}.
\end{remark}

\begin{theorem} 
  \label{thm:conv-diffu}
  Assume that $ u_0 \in D(A^\nu) $ with $ 0 < \nu \leqslant 1 $. Then
  \begin{equation}
    \label{eq:conv-diffu-inf}
    \nm{u-U}_{L^\infty(0,T;X)} \leqslant
    C_{\alpha,\sigma,\nu,T} J^{-\min\{\sigma\nu\alpha,1\}}
    \nm{u_0}_{D(A^\nu)}.
  \end{equation}
\end{theorem}

The main task of the rest of this section is to prove \cref{thm:conv-diffu}. To
this end, we proceed as follows. Assume that $ \lambda > 0 $. For any $ y \in
{}_0H^{\alpha/2}(0,T) $, define $ \Pi_\tau^\lambda y \in \mathcal W_\tau $ by
that
\begin{equation}
  \label{eq:def-Pi-diffu}
  \Dual{
    \big( \D_{0+}^\alpha+\lambda \big)
    \big( y-\Pi_\tau^\lambda y \big), w
  }_{{}^0H^{\alpha/2}(0,T)} = 0
  \quad\forall w \in \mathcal W_\tau.
\end{equation}
For each $ 1 \leqslant m \leqslant J $, define $ G_\lambda^m \in \mathcal W_\tau
$ by that $ G_\lambda^m|_{(t_m,T)} = 0 $ and  
\begin{equation}
  \label{eq:def-G-diffu}
  \Dual{w, (\D_{t_m-}^\alpha + \lambda) G_\lambda^m}_{(0,t_m)} =
  \frac1{t_m-t_{m-1}} \int_{t_{m-1}}^{t_m} w
\end{equation}
for all $ w \in \mathcal W_\tau $. In addition, let $ G_{\lambda,m+1}^m := 0 $
and, for each $ 1 \leqslant j \leqslant m $, let
\[
  G_{\lambda,j}^m := \lim_{t \to {t_j-}} G_\lambda^m(t).
\]

\begin{remark}
The $G_{\lambda}^m$ can be viewed as a regularized Green function with respect to the operator $\D_{t_m-}^\alpha+\lambda$.
\end{remark}

\begin{lemma}
  \label{lem:G}
  For each $ 1 \leqslant m \leqslant J $,
  \begin{align}
    & G_{\lambda,m}^m > G_{\lambda,m-1}^m >
    \ldots > G_{\lambda,1}^m > 0, \label{eq:G-1} \\
    & G_{\lambda,m}^m = \frac1{
      (t_m-t_{m-1})^{1-\alpha}/\Gamma(2-\alpha) +
      \lambda(t_m-t_{m-1})
    }, \label{eq:G-2} \\
    & G_{\lambda,m}^m = \sum_{j=1}^{m-1}
    (G_{\lambda,j+1}^m - G_{\lambda,j}^m)
    \frac{
      t_j^{1-\alpha} - (t_j-t_1)^{1-\alpha}+\lambda\Gamma(2-\alpha) t_1
    }{
      t_m^{1-\alpha} - (t_m-t_1)^{1-\alpha}+\lambda \Gamma(2-\alpha) t_1
    }. \label{eq:G-3}
  \end{align}
\end{lemma}
\begin{proof}
  Let us first prove that
  \begin{equation}
    \label{eq:Gj}
    G_{\lambda,j+1}^m > G_{\lambda,j}^m
    \quad\text{ for all } 1 \leqslant j < m.
  \end{equation}
  For any $ 1 \leqslant k < m $, by \cref{eq:def-G-diffu} we obtain
  \begin{small}
  \[
    \sum_{j=k}^m (G_{\lambda,j}^m-G_{\lambda,j+1}^m) \big(
      (t_j-t_{k-1})^{1-\alpha} -
      (t_j-t_k)^{1-\alpha}
    \big) + \mu (t_k-t_{k-1}) G_{\lambda,k}^m = 0,
  \]
  \end{small}
  where $ \mu := \lambda \Gamma(2-\alpha) $, so that a simple algebraic
  computation yields
  \begin{small}
  \begin{equation}
    \label{eq:ava}
    \begin{aligned}
      & \sum_{j=k}^{m-1} (G_{\lambda,j+1}^m - G_{\lambda,j}^m) \big(
        (t_j-t_{k-1})^{1-\alpha} -
        (t_j-t_k)^{1-\alpha} + \mu(t_k-t_{k-1})
      \big) \\
      ={} &
      G_{\lambda,m}^m \big(
        (t_m-t_{k-1})^{1-\alpha} -
        (t_m-t_k)^{1-\alpha} + \mu (t_k-t_{k-1})
      \big).
    \end{aligned}
  \end{equation}
  \end{small}
  Inserting $ k=m-1 $ into the above equation and noting the fact $
  G_{\lambda,m}^m > 0 $ indicate $ G_{\lambda,m}^m>G_{\lambda,m-1}^m $. Assume that
  $ G_{\lambda,j+1}^m>G_{\lambda,j}^m $ for all $ k \leqslant j < m $, where $ 2
  \leqslant k < m $. Multiplying both sides of \cref{eq:ava} by
  \begin{small}
  \[
    \frac{
      (t_m-t_{k-2})^{1-\alpha} - (t_m-t_{k-1})^{1-\alpha} +
      \mu(t_{k-1}-t_{k-2})
    }{
      (t_m-t_{k-1})^{1-\alpha} - (t_m-t_{k})^{1-\alpha} +
      \mu(t_{k}-t_{k-1})
    },
  \]
  \end{small}
  from \cref{lem:pre-G} we obtain
  \begin{small}
  \begin{align*}
    & \sum_{j=k}^{m-1} (G_{\lambda,j+1}^m - G_{\lambda,j}^m) \big(
      (t_j-t_{k-2})^{1-\alpha} -
      (t_j-t_{k-1})^{1-\alpha} + \mu(t_{k-1}-t_{k-2})
    \big) \\
    <{} &
    G_{\lambda,m}^m \big(
      (t_m-t_{k-2})^{1-\alpha} -
      (t_m-t_{k-1})^{1-\alpha} + \mu (t_{k-1}-t_{k-2})
    \big).
  \end{align*}
  \end{small}
  Similarly to \cref{eq:ava}, we have
  \begin{small}
  \begin{align*}
    & \sum_{j=k-1}^{m-1} (G_{\lambda,j+1}^m - G_{\lambda,j}^m) \big(
      (t_j-t_{k-2})^{1-\alpha} -
      (t_j-t_k-1)^{1-\alpha} + \mu(t_{k-1}-t_{k-2})
    \big) \\
    ={} &
    G_{\lambda,m}^m \big(
      (t_m-t_{k-2})^{1-\alpha} -
      (t_m-t_{k-1})^{1-\alpha} + \mu (t_{k-1}-t_{k-2})
    \big).
  \end{align*}
  \end{small}
  Combining the above two equations yields $ G_{\lambda,k}^m>G_{\lambda,k-1}^m
  $. Therefore, \cref{eq:Gj} is proved by induction.

  Next, inserting $ k=1 $ into \cref{eq:ava} yields
  \begin{small}
  \begin{equation}
    \label{eq:Gjm}
    \begin{aligned}
      & \sum_{j=1}^{m-1} (G_{\lambda,j+1}^m - G_{\lambda,j}^m) \big(
        t_j^{1-\alpha} -
        (t_j-t_1)^{1-\alpha} + \mu t_1
      \big) \\
      ={} &
      G_{\lambda,m}^m \big(
        t_m^{1-\alpha} - (t_m-t_1)^{1-\alpha} + \mu t_1
      \big).
    \end{aligned}
  \end{equation}
  \end{small}
  Since
  \begin{small}
  \[
    t_j^{1-\alpha} - (t_j-t_1)^{1-\alpha} + \mu t_1 >
    t_m^{1-\alpha} - (t_m-t_1)^{1-\alpha} + \mu t_1
    \quad\forall 1 \leqslant j \leqslant m-1,
  \]
  \end{small}
  from \cref{eq:Gj,eq:Gjm} it follows that
  \[
    \sum_{j=1}^{m-1} \left(
      G_{\lambda,j+1}^m - G_{\lambda,j}^m
    \right) < G_{\lambda,m}^m.
  \]
  This implies $ G_{\lambda,1}^m > 0 $ and hence proves \cref{eq:G-1} by
  \cref{eq:Gj}.

  Finally, \cref{eq:G-2} is evident by \cref{eq:def-G-diffu}, and dividing both
  sides of \cref{eq:Gjm} by $ t_m^{1-\alpha}-(t_m-t_1)^{1-\alpha}+\mu t_1 $
  proves \cref{eq:G-3}. This completes the proof.
\end{proof}

\begin{lemma} 
  For each $ 1 \leqslant k \leqslant J $,
  \begin{small}
  \begin{align} 
    & \sum_{j=1}^k j^{(\sigma-1)(\alpha-1)}
    \nm{(I-Q_\tau)(t_k-t)^{-\alpha}}_{L^1(t_{j-1},t_j)}
    \leqslant C_{\alpha,\sigma,T} J^{\sigma(\alpha-1)},
    \label{eq:74} \\
    & \sum_{j=1}^k j^{-\sigma-\alpha+1}
    \nm{(I-Q_\tau)(t_k-t)^{-\alpha}}_{L^1(t_{j-1},t_j)}
    \leqslant C_{\alpha,\sigma,T}
    J^{\sigma(\alpha-1)} k^{-\sigma\alpha}.
    \label{eq:73}
  \end{align}
  \end{small}
\end{lemma}
\begin{proof} 
  A straightforward calculation gives
  \begin{align*}
    & k^{(\sigma-1)(\alpha-1)}
    \nm{(I-Q_\tau)(t_k-t)^{-\alpha}}_{L^1(t_{k-1},t_k)} \\
    \leqslant{} &
    C_{\alpha} k^{(\sigma-1)(\alpha-1)}
    (t_k-t_{k-1})^{1-\alpha} \\
    \leqslant{} &
    C_{\alpha,\sigma,T} J^{-\sigma(1-\alpha)}
  \end{align*}
  and
  \begin{small}
  \begin{align*}
    & \sum_{j=1}^{k-1} j^{(\sigma-1)(\alpha-1)}
    \nm{(I-Q_\tau)(t_k-t)^{-\alpha}}_{L^1(t_{j-1},t_j)} \\
    \leqslant{} &
    C_{\alpha} \sum_{j=1}^{k-1}
    j^{(\sigma-1)(\alpha-1)} (t_j-t_{j-1})
    \big( (t_k-t_j)^{-\alpha}-(t_k-t_{j-1})^{-\alpha} \big) \\
    \leqslant{} &
    C_{\alpha,\sigma,T} J^{-\sigma(1-\alpha)}
    \sum_{j=1}^{k-1} j^{(\sigma-1)(\alpha-1)}
    \big( j^\sigma-(j-1)^\sigma \big) \big(
      (k^\sigma-j^\sigma)^{-\alpha} -
      (k^\sigma-(j-1)^\sigma)^{-\alpha}
    \big) \\
    \leqslant{} &
    C_{\alpha,\sigma,T} J^{-\sigma(1-\alpha)}
    \sum_{j=1}^{k-1} j^{(\sigma-1)(\alpha-1)}
    j^{2(\sigma-1)}(k^\sigma-j^\sigma)^{-\alpha-1} \\
    ={} &
    C_{\alpha,\sigma,T} J^{-\sigma(1-\alpha)}
    \sum_{j=1}^{k-1} j^{(\sigma-1)(\alpha+1)}
    (k^\sigma - j^\sigma)^{-\alpha-1} \\
    \leqslant{} &
    C_{\alpha,\sigma,T} J^{-\sigma(1-\alpha)}
    \quad\text{(by \cref{lem:discr-conv1}).}
  \end{align*}
  \end{small}
  Combining the above two estimates proves \cref{eq:74}. Similarly, a simple
  calculation gives
  \begin{align*} 
    & k^{-\sigma-\alpha+1}
    \nm{(I-Q_\tau)(t_k-t)^{-\alpha}}_{L^1(t_{k-1},t_k)} \\
    \leqslant{} &
    C_\alpha k^{-\sigma-\alpha+1} (t_k-t_{k-1})^{1-\alpha} \\
    \leqslant{} &
    C_{\alpha,\sigma,T} J^{-\sigma(1-\alpha)} k^{-\sigma\alpha}
  \end{align*}
  and
  \begin{small}
  \begin{align*} 
    & \sum_{j=1}^{k-1} j^{-\sigma-\alpha+1}
    \nm{(I-Q_\tau)(t_k-t)^{-\alpha}}_{L^1(t_{j-1},t_j)} \\
    \leqslant{} &
    C_{\alpha} \sum_{j=1}^{k-1}
    j^{-\sigma-\alpha+1}(t_j-t_{j-1})
    \big(
      (t_k-t_j)^{-\alpha} - (t_k-t_{j-1})^{-\alpha}
    \big) \\
    \leqslant{} &
    C_{\alpha,\sigma,T} J^{-\sigma(1-\alpha)}
    \sum_{j=1}^{k-1} j^{-\sigma-\alpha+1}
    \big(j^\sigma-(j-1)^\sigma\big) \big(
      (k^\sigma - j^\sigma)^{-\alpha} -
      (k^\sigma - (j-1)^\sigma)^{-\alpha}
    \big) \\
    \leqslant{} &
    C_{\alpha,\sigma,T} J^{-\sigma(1-\alpha)}
    \sum_{j=1}^{k-1} j^{-\sigma-\alpha+1}
    j^{2(\sigma-1)} (k^\sigma - j^\sigma)^{-\alpha-1} \\
    ={} &
    C_{\alpha,\sigma,T} J^{-\sigma(1-\alpha)}
    \sum_{j=1}^{k-1} j^{\sigma-\alpha-1}
    (k^\sigma - j^\sigma)^{-\alpha-1} \\
    \leqslant{} &
    C_{\alpha,\sigma,T} J^{-\sigma(1-\alpha)} k^{-\sigma\alpha}
    \quad\text{(by \cref{lem:discr-conv1}).}
  \end{align*}
  \end{small}
  Combining the above two estimates proves \cref{eq:73} and thus concludes the
  proof.
\end{proof}

\begin{lemma} 
		 \label{lem:wg}
  For each $ 1 \leqslant m \leqslant J $,
  \begin{equation}
    \label{eq:weigh-G}
    \sum_{j=1}^m  \left( \frac mj \right)^{(\sigma-1)(1-\alpha)}
    \nm{(I-Q_\tau)\D_{t_m-}^\alpha G_\lambda^m}_{L^1(t_{j-1},t_j)} \leqslant
    C_{\alpha,\sigma,T}.
  \end{equation}
\end{lemma}
\begin{proof} 
  For each $ 1 \leqslant j \leqslant m $, let
  \begin{equation}
    \label{eq:eta}
    \eta_j^m := \frac{
      \big( J/j \big)^{\sigma\alpha} + \lambda
    }{(J/m)^{\sigma\alpha}+\lambda} j^{(\sigma-1)(\alpha-1)} J^{\sigma(1-\alpha)}.
  \end{equation}
  Since
  \[
    (\D_{t_m-}^\alpha G_\lambda^m)(t) = \sum_{j=1}^m
    (G_{\lambda,j}^m-G_{\lambda,j+1}^m)
    \frac{(t_j-t)_{+}^{-\alpha}}{\Gamma(1-\alpha)},
  \]
  we have
  \begin{small}
  \begin{align*}
    & \sum_{j=1}^m \eta_j^m
    \nm{(I-Q_\tau)\D_{t_m-}^\alpha G_\lambda^m}_{L^1(t_{j-1},t_j)} \\
    \leqslant{} &
    \frac1{\Gamma(1-\alpha)} \sum_{j=1}^m \eta_j^m \sum_{k=j}^m
    \snm{G_{\lambda,k}^m-G_{\lambda,k+1}^m} \nm{(I-Q_\tau)(t_k-t)^{-\alpha}}_{L^1(t_{j-1},t_j)} \\
    ={} &
    \frac1{\Gamma(1-\alpha)} \sum_{k=1}^m
    \snm{G_{\lambda,k}^m-G_{\lambda,k+1}^m} \sum_{j=1}^k \eta_j^m
    \nm{(I-Q_\tau)(t_k-t)^{-\alpha}}_{L^1(t_{j-1},t_j)} \\
    \leqslant{} &
    C_{\alpha,\sigma,T} \sum_{k=1}^m \snm{
      G_{\lambda,k}^m - G_{\lambda,k+1}^m
    } \frac{J^{\sigma(1-\alpha)}}{(J/m)^{\sigma\alpha}+\lambda} \\
    & \qquad \times \sum_{j=1}^k
    \big( (J/j)^{\sigma\alpha} + \lambda \big)
    j^{(\sigma-1)(\alpha-1)}
    \nm{(I-Q_\tau)(t_k-t)^{-\alpha}}_{L^1(t_{j-1},t_j)} \\
    \leqslant{} &
    C_{\alpha,\sigma,T} \sum_{k=1}^m
    \frac{
      \big( J/k \big)^{\sigma\alpha} + \lambda
    }{(J/m)^{\sigma\alpha}+\lambda}
    \snm{G_{\lambda,k}^m-G_{\lambda,k+1}^m}
    \quad\text{(by \cref{eq:74,eq:73}).}
  \end{align*}
  \end{small}
  Therefore, from \cref{lem:G} and the inequality
  \[
    \frac{
      t_k^{1-\alpha} - (t_k-t_1)^{1-\alpha} + \lambda \Gamma(2-\alpha) t_1
    }{
      t_m^{1-\alpha} - (t_m-t_1)^{1-\alpha}+\lambda\Gamma(2-\alpha)t_1
    } \geqslant C_{\alpha,\sigma,T}
    \frac{
      \big( J/k \big)^{\sigma\alpha}+\lambda
    }{(J/m)^{\sigma\alpha}+\lambda},
  \]
  it follows that
  \[
    \sum_{j=1}^m \eta_j^m \nm{
      (I-Q_\tau)\D_{t_m-}^\alpha G_\lambda^m
    }_{L^1(t_{j-1},t_j)}
    \leqslant C_{\alpha,\sigma,T} G_{\lambda,m}^m.
  \]
  In addition, by \cref{eq:G-2,eq:eta}, it holds
  \begin{small}
  \begin{align*}
    \frac{\eta_j^m}{G_{\lambda,m}^m} & \geqslant
    C_{\alpha,\sigma,T} \frac{
      (J/j)^{\sigma\alpha}+\lambda
    }{(J/m)^{\sigma\alpha}+\lambda}
    j^{(\sigma-1)(\alpha-1)} J^{\sigma(1-\alpha)}
    \big(
      m^{(\sigma-1)(1-\alpha)}J^{\sigma(\alpha-1)} +
      \lambda m^{\sigma-1} J^{-\sigma}
    \big) \\
    & \geqslant
    C_{\alpha,\sigma,T} \frac{
      \big(J/j\big)^{\sigma\alpha}+\lambda
    }{(J/m)^{\sigma\alpha}+\lambda}
    j^{(\sigma-1)(\alpha-1)}
    m^{(\sigma-1)(1-\alpha)} \\
    & \geqslant
    C_{\alpha,\sigma,T}
    \big(m/j\big)^{(\sigma-1)(1-\alpha)}.
  \end{align*}
  \end{small}
  Consequently, combining the above two estimates proves \cref{eq:weigh-G} and
  thus concludes the proof.
\end{proof}
\begin{remark}
  $\D_{t_m-}^\alpha  G^m_{\lambda}$ is a non-smooth function in $L^1(0,T)$, but it is smoother   away from $t_m$. This is the starting point of \cref{lem:wg}.
\end{remark}

\begin{lemma} 
  \label{lem:love}
  If $ y \in {}_0H^{\alpha/2}(0,T) \cap C(0,T] $, then
  \begin{equation}
    \label{eq:love}
    \left(
      \Pi_\tau^\lambda y - Q_\tau y
    \right)({t_m-}) = \Dual{
      (I-Q_\tau)y, \, (I-Q_\tau)\D_{t_m-}^\alpha G_\lambda^m
    }_{(0,t_m)}
  \end{equation}
  for each $ 1 \leqslant m \leqslant J $. 
\end{lemma}
\begin{proof} 
  A straightforward calculation gives
  \begin{align*}
    & (\Pi_\tau^\lambda y - Q_\tau y)(t_m-) \\
    ={}& \Dual{
      \Pi_\tau^\lambda y - Q_\tau y, \,
      (\D_{t_m-}^\alpha+\lambda) G_\lambda^m
    }_{(0,t_m)} \qquad\qquad\text{(by \cref{eq:def-G-diffu})} \\
    ={}& \Dual{
      \Pi_\tau^\lambda y - Q_\tau y, \,
      (\D_{T-}^\alpha+\lambda) G_\lambda^m
    }_{(0,T)} \qquad\qquad\text{(by the fact $ G_\lambda^m|_{(t_m,T)}=0 $)} \\
    ={} &
    \Dual{
      (\D_{0+}^\alpha + \lambda)(\Pi_\tau^\lambda y - Q_\tau y), \,
      G_\lambda^m
    }_{{}^0H^{\alpha/2}(0,T)} \qquad\text{(by \cref{lem:dual})} \\
    ={} &
    \Dual{
      (\D_{0+}^\alpha + \lambda)(I-Q_\tau)y, \,
      G_\lambda^m
    }_{{}^0H^{\alpha/2}(0,T)} \qquad\text{(by \cref{eq:def-Pi-diffu})} \\
    ={} &
    \Dual{
      (I-Q_\tau)y, \,
      (\D_{T-}^\alpha + \lambda)G_\lambda^m
    }_{(0,T)} \qquad\qquad\text{(by \cref{lem:dual})} \\
    ={} &
    \Dual{
      (I-Q_\tau)y, \,
      (\D_{t_m-}^\alpha + \lambda)G_\lambda^m
    }_{(0,t_m)} \qquad\qquad
    \text{(by the fact $ G_\lambda^m|_{(t_m,T)} = 0 $).}
  \end{align*}
  Hence, \cref{eq:love} follows from the equality
  \[
    \Dual{
      (I-Q_\tau) y, (\D_{t_m-}^\alpha + \lambda)G_\lambda^m
    }_{(0,t_m)} = \Dual{
      (I-Q_\tau)y, (I-Q_\tau)\D_{t_m-}^\alpha G_\lambda^m
    }_{(0,t_m)},
  \]
  which is easily derived by the definition of $ Q_\tau $. This completes the
  proof.
\end{proof}


\begin{lemma} 
  \label{lem:err-Pi}
  Assume that 
  $ y \in {}_0H^{\alpha/2}(0,T) \cap C^1(0,T] $
  satisfies
  \begin{equation}
    \label{eq:y'-cond}
    \snm{y'(t)} \leqslant t^{-r},
    \quad 0 < t \leqslant T,
  \end{equation}
  where $ 0 < r < 1 $. Then
  \begin{equation}
    \label{eq:err-Pi-infty}
    \Nm{
      \big(I-\Pi_\tau^\lambda\big)y
    }_{L^\infty(0,T)} \leqslant
    C_{\alpha,\sigma,r,T} J^{-\min\{\sigma(1-r),1\}}.
  \end{equation}
\end{lemma}
\begin{proof} 
  For any $ 1 \leqslant m \leqslant J $,
  \begin{small}
  \begin{align*} 
    & \Snm{
      \big(\Pi_\tau^\lambda y - Q_\tau y\big)(t_m-)
    } \\
    ={}& \Snm{
      \Dual{
        (I-Q_\tau) y, (I-Q_\tau) \D_{t_m-}^\alpha G_\lambda^m
      }_{(0,t_m)}
    } \quad \text{(by \cref{lem:love})} \\
    \leqslant{} &
    \sum_{j=1}^m \nm{(I-Q_\tau)y}_{L^\infty(t_{j-1},t_j)}
    \nm{
      (I-Q_\tau)\D_{t_m-}^\alpha G_\lambda^m
    }_{L^1(t_{j-1},t_j)} \\
    \leqslant{} &
    \max_{1 \leqslant j \leqslant m } (m/j)^{(\sigma-1)(\alpha-1)}
    \nm{(I-Q_\tau)y}_{L^\infty(t_{j-1},t_j)} \\
    & \quad {} \times \sum_{j=1}^m (m/j)^{(\sigma-1)(1-\alpha)}
    \nm{(I-Q_\tau)\D_{t_m-}^\alpha G_\lambda^m}_{L^1(t_{j-1},t_j)} \\
    \leqslant{} &
    C_{\alpha,\sigma,T} \max_{1 \leqslant j \leqslant m}
    (m/j)^{(\sigma-1)(\alpha-1)}
    \nm{(I-Q_\tau)y}_{L^\infty(t_{j-1},t_j)}
    \quad\text{(by \cref{eq:weigh-G})} \\
    \leqslant{} &
    C_{\alpha,\sigma,T} \nm{(I-Q_\tau)y}_{L^\infty(0,t_m)}.
  \end{align*}
  \end{small}
  It follows that
  \begin{small}
  \begin{align*}
    \bnm{
      \big(\Pi_\tau^\lambda - Q_\tau\big)y
    }_{L^\infty(0,T)} =
    \max_{1 \leqslant m \leqslant J}
    \Snm{
      \big(\Pi_\tau^\lambda y - Q_\tau y\big)(t_m-)
    }
    \leqslant C_{\alpha,\sigma,T}
    \nm{(I-Q_\tau)y}_{L^\infty(0,T)},
  \end{align*}
  \end{small}
  and hence
  \begin{small}
  \begin{align*}
    \bnm{\big(I-\Pi_\tau^\lambda\big)y}_{L^\infty(0,T)}
    & \leqslant \nm{(I-Q_\tau)y}_{L^\infty(0,T)} +
    \Nm{
      \big(\Pi_\tau^\lambda - Q_\tau\big)y
    }_{L^\infty(0,T)} \\
    & \leqslant C_{\alpha,\sigma,T}
    \nm{(I-Q_\tau)y}_{L^\infty(0,T)}.
  \end{align*}
  \end{small}
  In addition, by \cref{eq:y'-cond} we obtain
  \begin{small}
  \begin{align*}
    & \nm{(I-Q_\tau)y}_{L^\infty(0,T)} \leqslant
    \max_{1 \leqslant j \leqslant J}
    \nm{(I-Q_\tau)y}_{L^\infty(t_{j-1},t_j)} \\
    \leqslant{} &
    \max_{ 1 \leqslant j \leqslant J }
    \big( t_j^{1-r}-t_j^{1-r} \big)/(1-r) \\
    \leqslant{} & C_{\alpha,\sigma,r,T}
    \max_{1 \leqslant j \leqslant J}
    j^{\sigma(1-r)-1} J^{-\sigma(1-r)} \\
    \leqslant{} & C_{\alpha,\sigma,r,T}
    J^{-\min\{\sigma(1-r),1\}}.
  \end{align*}
  \end{small}
  Finally, combining the above two estimates proves \cref{eq:err-Pi-infty} and
  hence this lemma.
\end{proof}

Finally, we are in a position to prove \cref{thm:conv-diffu} as follows. \\
\noindent {\bf Proof of \cref{thm:conv-diffu}.}
For each $ n \in \mathbb N $, let
\[
  u^n(t) := (u(t), \phi_n)_X, \quad  0 \leqslant t \leqslant T.
\]
By \cref{eq:u} we have
\[
  u^n(t) = E_{\alpha,1}(-\lambda_n t^\alpha)
  (u_0, \phi_n)_X, \quad 0 < t \leqslant T.
\]
A straightforward calculation gives
\[
  (u^n)'(t) = -\lambda_n t^{\alpha-1}
  E_{\alpha,\alpha}(-\lambda_n t^\alpha)
  (u_0,\phi_n)_X \phi_n, \quad 0 \leqslant t \leqslant T,
\]
and hence \cref{eq:ml_grow} implies
\begin{equation}
  \label{eq:un'}
  \snm{(u^n)'(t)} \leqslant C_\alpha
  t^{\nu\alpha-1} \lambda_n^\nu \snm{(u_0,\phi_n)_X},
  \quad 0 < t \leqslant T.
\end{equation}
By \cref{eq:U-diffu,eq:u-weak-diffu,eq:def-Pi-diffu} we have $ U =
\sum_{n=0}^\infty (\Pi_\tau^{\lambda_n} u^n) \phi_n $, so that
\begin{align*}
  & \nm{u-U}_{L^\infty(0,T;X)} =
  \sup_{0 < t < T} \Big(
    \sum_{n=0}^\infty
    \snm{(u^n - \Pi_\tau^{\lambda_n} u^n)(t)}^2
  \Big)^{1/2}  \\
  \leqslant{} &
  \Big(
    \sum_{n=0}^\infty \nm{
      (I-\Pi_\tau^{\lambda_n})u^n
    }_{L^\infty(0,T)}^2
  \Big)^{1/2} \\
  \leqslant{} &
  C_{\alpha,\sigma,T} J^{-\min\{\sigma\alpha\nu,1\}}
  \Big(
    \sum_{n=0}^\infty \lambda_n^{2\nu} (u_0,\phi_n)_X^2
  \Big)^{1/2} \quad\text{(by \cref{lem:err-Pi,eq:un'})} \\
  ={} &
  C_{\alpha,\sigma,T} J^{-\min\{\sigma\alpha\nu,1\}}
  \nm{u_0}_{D(A^\nu)}.
\end{align*}
This proves \cref{eq:conv-diffu-inf} and thus concludes the proof.
\hfill\ensuremath{\blacksquare}

\section{Fractional wave equation (\texorpdfstring{$ 1 < \alpha < 2 $}))}
\label{sec:wave}
This section considers the following discretization: seek $ U \in
W_\tau^\text{c} $ such that $ U(0) = u_0 $ and
\begin{equation}
  \label{eq:U-wave}
  \int_0^T \left( \D_{0+}^{\alpha-1}U' + AU, V \right)_X
  \, \mathrm{d}t = 0 \quad \forall V \in W_\tau. 
\end{equation}
\begin{remark}
  By \cref{eq:u}, a straightforward calculation gives that
  \begin{equation}
    \label{eq:u-weak-wave}
    \int_0^T (\D_{0+}^{\alpha-1} u' + Au, V)_X \, \mathrm{d}t = 0
    \quad \forall V \in W_\tau.
  \end{equation}
\end{remark}

\begin{remark} 
 We note that  when using uniform temporal grids, the  discretization \cref{eq:U-wave} is equivalent to
the L1 scheme (cf. \cite{Li-Wang-Xie2019Wave}),
\end{remark}

\begin{theorem} 
  \label{thm:conv-wave}
  Assume that $ u_0 \in D(A^\nu) $ with $ 1/2 < \nu \leqslant 1 $. If
  \begin{equation}
    \label{eq:sigma-cond-wave}
    \sigma > \frac{3-\alpha}{\alpha(\nu-1/2)},
  \end{equation}
  then
  \begin{equation}
    \label{eq:conv-wave-inf}
    \max_{1 \leqslant m \leqslant J}
    \nm{(u-U)(t_m)}_X \leqslant
    C_{\alpha,\sigma,T} J^{\alpha-3}
    \nm{u_0}_{D(A^\nu)}.
  \end{equation}
\end{theorem}

The main task of the rest of this section is to prove the theorem  above. For
each $ 1 \leqslant m \leqslant J $, define $ \mathcal G^m \in \mathcal W_\tau $
by that $ \mathcal G^m|_{(t_m,T)} = 0 $ and that
\begin{equation}
  \label{eq:def-G-wave}
  \Dual{w,\D_{t_m-}^{\alpha-1}\mathcal G^m}_{(0,t_m)} =
  \Dual{1, w}_{(0,t_m)} \quad \forall w \in \mathcal W_\tau.
\end{equation}
Let $ \mathcal G_{m+1}^m = 0 $ and, for each $ 1 \leqslant j \leqslant m $, let
\[
  \mathcal G_j^m := \lim_{t \to {t_j-}} \mathcal G^m(t).
\]
Since
\[
  \D_{t_m-}^{\alpha-1}\mathcal G^m =
  \sum_{j=1}^m (\mathcal G_j^m - \mathcal G_{j+1}^m)
  \frac{(t_j-t)_{+}^{1-\alpha}}{\Gamma(2-\alpha)},
\]
a straightforward calculation yields, from \cref{eq:def-G-wave}, that
\begin{equation}
  \label{eq:90}
  \sum_{j=k}^m \left(
    \mathcal G_j^m - \mathcal G_{j+1}^m
  \right) \left(
    (t_j-t_{k-1})^{2-\alpha} -
    (t_j-t_k)^{2-\alpha}
  \right) = \Gamma(3-\alpha)(t_k - t_{k-1})
\end{equation}
for each $ 1 \leqslant k \leqslant m $.

\begin{remark}
	Although $\mathcal{G}^m$ is not a regularized Green function, it has similar properties. 
\end{remark}
\begin{lemma} 
  \label{lem:lbj}
  For any $ 1/2 < \beta < 1 $ and $ 1 \leqslant k \leqslant J $,
  \begin{small}
  \begin{equation}
    \label{eq:lbj}
    \sum_{j=1}^k (j/J)^{\sigma(1-\alpha)} \big(
      (t_k-t_{j-1})^{1-\beta} - (t_k-t_j)^{1-\beta}
    \big) \leqslant
    C_{\alpha,\sigma,T}
    (k/J)^{\sigma(2-\alpha-\beta)}.
  \end{equation}
  \end{small}
\end{lemma}
\begin{proof} 
  An elementary calculation gives
  \begin{small}
  \[
    k^{\sigma(1-\alpha)}
    \big( k^\sigma - (k-1)^\sigma \big)^{1-\beta}
    \leqslant C_\sigma k^{\sigma(1-\alpha)}
    k^{(\sigma-1)(1-\beta)} =
    C_\sigma k^{\sigma(2-\alpha-\beta)+\beta-1}
  \]
  \end{small}
  and
  \begin{small}
  \begin{align*}
    & \sum_{j=1}^{k-1} j^{\sigma(1-\alpha)}
    \left(
      \big( k^\sigma-(j-1)^\sigma \big)^{1-\beta} -
      \big( k^\sigma-j^\sigma \big)^{1-\beta}
    \right) \\
    \leqslant{} &
    C_\sigma (1-\beta) \sum_{j=1}^{k-1}
    j^{\sigma(1-\alpha)} \big( k^\sigma-j^\sigma \big)^{-\beta}
    j^{\sigma-1} \\
    ={} & C_\sigma (1-\beta) \sum_{j=1}^{k-1}
    j^{2\sigma-\sigma\alpha-1}\big( k^\sigma-j^\sigma \big)^{-\beta} \\
    \leqslant{} &
    C_{\alpha,\sigma} k^{\sigma(2-\alpha-\beta)}
    \quad\text{(by \cref{lem:discr-conv2}).}
  \end{align*}
  \end{small}
  It follows that
  \begin{small}
  \begin{align*}
    & \sum_{j=1}^k (j/J)^{\sigma(1-\alpha)} \big(
      (t_k-t_{j-1})^{1-\beta} - (t_k-t_j)^{1-\beta}
    \big) \\
    ={} &
    J^{-\sigma(2-\alpha-\beta)}
    T^{1-\beta} \sum_{j=1}^k j^{\sigma(1-\alpha)} \left(
      \big( k^\sigma-(j-1)^\sigma \big)^{1-\beta} -
      \big( k^\sigma - j^\sigma \big)^{1-\beta}
    \right) \\
    \leqslant{} &
    C_{\alpha,\sigma,T} J^{-\sigma(2-\alpha-\beta)}
    \left(
      k^{\sigma(2-\alpha-\beta)+\beta-1} +
      k^{\sigma(2-\alpha-\beta)}
    \right) \\
    \leqslant{} &
    C_{\alpha,\sigma,T} (k/J)^{\sigma(2-\alpha-\beta)}.
  \end{align*}
  \end{small}
  This proves \cref{eq:lbj} and hence this lemma.
\end{proof}

\begin{lemma} 
  \label{lem:G2}
  For any $ 1/2 < \beta < 1 $ and $ 1 \leqslant m \leqslant J $,
  \begin{equation}
    \label{eq:G2}
    \sum_{j=1}^m (j/J)^{\sigma(1-\alpha)}
    \nm{\D_{t_m-}^\beta \mathcal G^m}_{L^1(t_{j-1},t_j)}
    \leqslant C_{\alpha,\sigma,T}.
  \end{equation}
\end{lemma}
\begin{proof}
 By \cref{eq:90,lem:G-wave},  an inductive argument yields  that
  \begin{equation}
    \label{eq:92}
    \mathcal G_1^m > \mathcal G_2^m > \ldots > \mathcal G_m^m =
    \Gamma(3-\alpha) (t_m-t_{m-1})^{\alpha-1}.
  \end{equation}
  Plugging $ k = 1 $ into \cref{eq:90} shows
  \begin{align*}
    \sum_{j=1}^m \left(
      \mathcal G_j^m - \mathcal G_{j+1}^m
    \right) \big(
      t_j^{2-\alpha} - (t_j-t_1)^{2-\alpha}
    \big) = \Gamma(3-\alpha) t_1,
  \end{align*}
  and hence
  \[
    \sum_{j=1}^m \frac{
      t_j^{2-\alpha} - (t_j-t_1)^{2-\alpha}
    }{\Gamma(3-\alpha)t_1} \left(
      \mathcal G_j^m - \mathcal G_{j+1}^m
    \right) = 1.
  \]
  From  \cref{eq:92} and the inequality
  \[
    \frac{t_j^{2-\alpha}-(t_j-t_1)^{2-\alpha}}
    {\Gamma(3-\alpha)t_1}
    \geqslant C_{\alpha,\sigma,T}
    (j/J)^{\sigma(1-\alpha)},
  \]
  it follows that
  \begin{equation}
    \label{eq:91}
    \sum_{j=1}^m (j/J)^{\sigma(1-\alpha)}
    \big(\mathcal G_j^m-\mathcal G_{j+1}^m\big)
    \leqslant C_{\alpha,\sigma,T}.
  \end{equation}
  Since
  \[
    \D_{t_m-}^\beta \mathcal G^m = \sum_{j=1}^m \left(
      \mathcal G_j^m-\mathcal G_{j+1}^m
    \right) \frac{(t_j-t)_{+}^{-\beta}}{\Gamma(1-\beta)},
  \]
  we obtain
  \begin{small}
  \begin{align*}
    & \sum_{j=1}^m (j/J)^{\sigma(1-\alpha)}
    \nm{\D_{t_m-}^\beta \mathcal G^m}_{L^1(t_{j-1}, t_j)} \\
    \leqslant{} &
    \sum_{j=1}^m (j/J)^{\sigma(1-\alpha)}
    \sum_{k=j}^m (\mathcal G_k^m-\mathcal G_{k+1}^m)
    \frac{
      (t_k-t_{j-1})^{1-\beta} - (t_k-t_j)^{1-\beta}
    }{\Gamma(2-\beta)} \quad\text{(by \cref{eq:92})} \\
    ={} &
    \sum_{k=1}^m (\mathcal G_k^m-\mathcal G_{k+1}^m) \sum_{j=1}^k
    (j/J)^{\sigma(1-\alpha)} \frac{
      (t_k-t_{j-1})^{1-\beta} - (t_k-t_j)^{1-\beta}
    }{\Gamma(2-\beta)} \\
    \leqslant{} & C_{\alpha,\sigma,T}
    \sum_{k=1}^m (k/J)^{\sigma(2-\alpha-\beta)}
    \left(\mathcal G_k^m - \mathcal G_{k+1}^m\right)
    \quad\text{(by \cref{lem:lbj,eq:92})} \\
    \leqslant{} & C_{\alpha,\sigma,T}
    \quad \text{(by \cref{eq:92,eq:91}).}
  \end{align*}
  \end{small}
  This proves \cref{eq:G2} and thus completes the proof.
\end{proof}
\begin{remark}
  For more details about proving \cref{eq:92}, we refer the reader to the proof
  of \cref{eq:G-1}.
\end{remark}

\begin{lemma} 
  \label{lem:I-Qtau-wave}
  Assume that $ y \in C^2((0,T];X) $ satisfies
  \begin{equation}
    \label{eq:732}
    t^{-1} \nm{y'(t)}_X + \nm{y''(t)}_X \leqslant t^{-r},
    \quad 0 < t \leqslant T,
  \end{equation}
  where $ 0 < r < 2 $. For each $ 1 \leqslant j \leqslant J $, the following
  three estimates hold: \\ if $ \sigma < 2/(3-r) $, then
  \begin{small}
  \begin{equation} 
    \label{eq:I-Qtau-wave-1}
    \Nm{
      \D_{0+}^{\alpha-2}(I\!-\!Q_\tau)y'
    }_{L^\infty(t_{j-1},t_j;X)}
    \leqslant C_{\alpha,\sigma,r,T}
    J^{-\sigma(3-\alpha-r)} j^{-\sigma\alpha}
    \big( j^{\sigma(3-r)+\alpha-3} + 1 \big);
  \end{equation}
  \end{small}
  if $ \sigma = 2/(3-r) $, then
  \begin{small}
  \begin{equation} 
    \label{eq:I-Qtau-wave-2}
    \Nm{
      \D_{0+}^{\alpha-2}(I\!-\!Q_\tau)y'
    }_{L^\infty(t_{j-1},t_j;X)}
    \leqslant C_{\alpha,r,T}
    J^{-\sigma(3-\alpha-r)} j^{-\sigma\alpha}
    \big( j^{\sigma(3-r)+\alpha-3} + \ln j \big);
  \end{equation}
  \end{small}
  if $ \sigma > 2/(3-r) $, then
  \begin{small}
  \begin{equation} 
    \label{eq:I-Qtau-wave}
    \Nm{
      \D_{0+}^{\alpha-2}(I\!-\!Q_\tau)y'
    }_{L^\infty(t_{j-1},t_j;X)}
    \leqslant C_{\alpha,\sigma,r,T}
    J^{-\sigma(3-\alpha-r)}
    j^{\sigma(3-\alpha-r)+\alpha-3}.
  \end{equation}
  \end{small}
\end{lemma}
\begin{proof} 
  We only present a proof of \cref{eq:I-Qtau-wave}, the proofs of
  \cref{eq:I-Qtau-wave-1,eq:I-Qtau-wave-2} being similar. Since the case $ r=1 $
  can be proved analogously, we assume that $ r \neq 1 $.

  Let us first prove that
  \begin{small}
  \begin{equation}
    \label{eq:a-t}
    \begin{aligned}
       \sup_{t_{j-1} \leqslant a < t_j}
      \Nm{
        \Dual{
          (a-t)^{1-\alpha}, \,
          (I\!-\!Q_\tau)y'
        }_{(0,t_{j-1})}
      }_X 
      \leqslant{}  C_{\alpha,\sigma,r,T}
      J^{-\sigma(3-\alpha-r)}
      j^{\sigma(3-\alpha-r)+\alpha-3}
    \end{aligned}
  \end{equation}
  \end{small}
  for each $ 2 \leqslant j \leqslant J $. Since the case $ j=2 $ can be easily
  verified, we assume that $ 3 \leqslant j \leqslant J $. Let $ t_{j-1}
  \leqslant a < t_j $. By the definition of $ Q_\tau $, we have
  \begin{equation}
    \label{eq:shit}
    \Nm{
      \Dual{(a-t)^{1-\alpha}, (I-Q_\tau)y'}_{(0,t_{j-1})}
    }_X \leqslant \mathbb I_1 + \mathbb I_2 + \mathbb I_3,
  \end{equation}
  where
  \begin{small}
  \begin{align*}
    \mathbb I_1 &:= \Nm{
      \Dual{
        (I-Q_\tau)(a-t)^{1-\alpha}, (I-Q_\tau)y'
      }_{(0,t_1)}
    }_X, \\
    \mathbb I_2 &:= \sum_{k=2}^{j-2}
    \Nm{
      \Dual{
        (I-Q_\tau)(a-t)^{1-\alpha},
        (I-Q_\tau)y'
      }_{(t_{k-1},t_k)}
    }_X, \\
    \mathbb I_3 &:= \Nm{
      \Dual{
        (I-Q_\tau)(a-t)^{1-\alpha},
        (I-Q_\tau)y'
      }_{(t_{j-2},t_{j-1})}
    }_X.
  \end{align*}
  \end{small}
  By \cref{eq:732} and the facts $ \sigma>2/(3-r) $ and $ t_{j-1} \leqslant a $,
  a routine calculation yields the following three estimates:
  \begin{small}
  \begin{align*}
    \mathbb I_1 & \leqslant
    \nm{(I-Q_\tau)(a-t)^{1-\alpha}}_{L^\infty(0,t_1)}
    \nm{(I-Q_\tau)y'}_{L^1(0,t_1;X)} \\
    & \leqslant C_{\alpha,r} \big(
      (a-t_1)^{1-\alpha}-a^{1-\alpha}
    \big) t_1^{2-r} \\
    & \leqslant C_{\alpha,r} \big(
      (t_{j-1}-t_1)^{1-\alpha}-t_{j-1}^{1-\alpha}
    \big) t_1^{2-r} \\
    & \leqslant C_{\alpha,\sigma,r,T} J^{-\sigma(3-\alpha-r)}
    \big(
      ((j-1)^\sigma-1)^{1-\alpha} -
      (j-1)^{\sigma(1-\alpha)}
    \big) \\
    & \leqslant C_{\alpha,\sigma,r,T} J^{-\sigma(3-\alpha-r)}
    j^{-\sigma\alpha} \\
    & \leqslant C_{\alpha,\sigma,r,T}
    J^{-\sigma(3-\alpha-r)}
    j^{\sigma(3-\alpha-r) + \alpha - 3},
  \end{align*}
  \begin{align*}
    \mathbb I_2 & \leqslant C_\alpha \sum_{k=2}^{j-2} 
    \nm{(I-Q_\tau) y'}_{L^\infty(t_{k-1},t_k;X)}
    (t_k-t_{k-1}) \big(
      (a-t_k)^{1-\alpha} - (a-t_{k-1})^{1-\alpha}
    \big) \\
    & \leqslant C_{\alpha,r}
    \sum_{k=2}^{j-2} \Snm{t_k^{1-r}-t_{k-1}^{1-r}} (t_k-t_{k-1})
    ((t_{j-1}-t_k)^{1-\alpha} - (t_{j-1}-t_{k-1})^{1-\alpha}) \\
    & \leqslant C_{\alpha,\sigma,r,T}
    J^{-\sigma(3-\alpha-r)} \sum_{k=2}^{j-2}
    k^{\sigma(1-r)-1} k^{2(\sigma-1)}
    (j^\sigma - k^\sigma)^{-\alpha} \\
    & = C_{\alpha,\sigma,r,T}
    J^{-\sigma(3-\alpha-r)} \sum_{k=2}^{j-2}
    k^{3\sigma-\sigma r-3} (j^\sigma-k^\sigma)^{-\alpha} \\
    & \leqslant C_{\alpha,\sigma,r,T}
    J^{-\sigma(3-\alpha-r)}
    j^{\sigma(3-\alpha- r) + \alpha - 3}
    \quad\text{(by \cref{lem:discr-conv1})}
  \end{align*}
  \end{small}
  and
  \begin{align*} 
    \mathbb I_3 & \leqslant
    C_\alpha \nm{(I-Q_\tau)y'}_{L^\infty(t_{j-2},t_{j-1};X)}
    \big( (a-t_{j-2})^{2-\alpha} - (a-t_{j-1})^{2-\alpha} \big) \\
    & \leqslant C_{\alpha,r} \Snm{ t_{j-1}^{1-r}-t_{j-2}^{1-r} }
    (t_{j-1}-t_{j-2})^{2-\alpha} \\
    & \leqslant
    C_{\alpha,\sigma,r,T}
    J^{-\sigma(3-\alpha-r)}
    j^{\sigma(1-r)-1} j^{(\sigma-1)(2-\alpha)} \\
    & =
    C_{\alpha,\sigma,r,T}
    J^{-\sigma(3-\alpha-r)}
    j^{\sigma(3-\alpha-r)+\alpha-3}.
  \end{align*}
  Since $a$, $ t_{j-1} \leqslant a < t_j $, is arbitrary, combining \cref{eq:shit} and
  the above three estimates proves   \cref{eq:a-t}   for   $ 3
  \leqslant j \leqslant J $.

  Next, let us prove that \cref{eq:I-Qtau-wave} holds for all $ 2 \leqslant j
  \leqslant J $. For any $ t_{j-1} \leqslant a < t_j $,
  \begin{equation}
    \label{eq:101}
    \big(\D_{0+}^{\alpha-2}(y'-Q_\tau y')\big)(a) =
    (\mathbb I_4 + \mathbb I_5) / \Gamma(2-\alpha),
  \end{equation}
  where
  \begin{align*}
    \mathbb I_4 &:= \Dual{
      (a-t)^{1-\alpha}, \, (I-Q_\tau)y'
    }_{(t_{j-1},a)}, \\
    \mathbb I_5 &:= \Dual{
      (a-t)^{1-\alpha}, \, (I-Q_\tau)y'
    }_{(0,t_{j-1})}.
  \end{align*}
  We have
  \begin{align*}
    \nm{\mathbb I_4}_X
    & \leqslant C_{\alpha} (a-t_{j-1})^{2-\alpha}
    \nm{(I-Q_\tau)y'}_{L^\infty(t_{j-1},t_j;X)} \\
    & \leqslant C_{\alpha,r} (t_j-t_{j-1})^{2-\alpha}
    \Snm{ t_j^{1-r} - t_{j-1}^{1-r} }
    \quad\text{(by \cref{eq:732})} \\
    & \leqslant C_{\alpha,\sigma,r,T} J^{-\sigma(3-\alpha-r)}
    j^{(\sigma-1)(2-\alpha)} j^{\sigma(1-r)-1} \\
    & = C_{\alpha,\sigma,r,T} J^{-\sigma(3-\alpha-r)}
    j^{\sigma(3-\alpha- r)+\alpha-3}
  \end{align*}
  and, by \cref{eq:a-t}, 
  \[
    \nm{\mathbb I_5}_X \leqslant C_{\alpha,\sigma,r,T}
    J^{-\sigma(3-\alpha-r)}
    j^{\sigma(3-\alpha-r)+\alpha-3}.
  \]
  Combining the above two estimates and \cref{eq:101} gives
  \[
    \nm{
      \big( \D_{0+}^{\alpha-2}(y'-Q_\tau y') \big)(a)
    }_X \leqslant C_{\alpha,\sigma,r,T}
    J^{-\sigma(3-\alpha-r)}
    j^{\sigma(3-\alpha-r)+\alpha-3}.
  \]
  Hence, the arbitrariness of $ t_{j-1} \leqslant a < t_j $ proves  
  \cref{eq:I-Qtau-wave}  for   $ 2 \leqslant j \leqslant J $.

  Finally, for any $ 0 < a \leqslant t_1 $,
  \begin{align*} 
    & \nm{
      \big( \D_{0+}^{\alpha-2} (I-Q_\tau)y' \big)(a)
    }_X \\
    \leqslant{} & C_{\alpha,r}
    \int_0^a (a-t)^{1-\alpha} (t^{1-r} + t_1^{1-r}) \, \mathrm{d}t
    \quad\text{(by \cref{eq:732})} \\
    \leqslant{} & C_{\alpha,r} \left(
      a^{3-\alpha-r} \int_0^1 (1-s)^{1-\alpha}s^{1-r} \, \mathrm{d}s +
      a^{2-\alpha} t_1^{1-r}
    \right) \\
    \leqslant{} & C_{\alpha,r}
    t_1^{3-\alpha-r} \leqslant
    C_{\alpha,\sigma,r,T} J^{-\sigma(3-\alpha-r)}.
  \end{align*}
  This proves \cref{eq:I-Qtau-wave} for $ j= 1 $ and thus concludes the proof.
\end{proof}

For any $ y \in H^{(\alpha+1)/2}(0,T) $, define $ \mathcal P_\tau y \in \mathcal
W_\tau^\text{c} $ by
\begin{equation}
  \label{eq:def-Ptau}
  \left\{
    \begin{aligned}
      & (y-\mathcal P_\tau y)(0) = 0, \\
      & \Dual{
        \D_{0+}^{\alpha-1}\big(y-\mathcal P_\tau y\big)', w
      }_{{}^0H^{(\alpha-1)/2}(0,T)} = 0
      \quad \forall w \in \mathcal W_\tau,
    \end{aligned}
  \right.
\end{equation}
and define $ \Xi_\tau^\lambda y \in \mathcal W_\tau^\text{c} $ by
\begin{equation}
  \label{eq:def-Pi-wave}
  \left\{
    \begin{aligned}
      & \left(y-\Xi_\tau^\lambda y\right)(0) = 0, \\
      & \Dual{
        \D_{0+}^{\alpha-1}\left(y-\Xi_\tau^\lambda y\right)' +
        \lambda \left(y-\Xi_\tau^\lambda y\right), w
      }_{{}^0H^{(\alpha-1)/2}(0,T)} = 0
      \quad \forall w \in \mathcal W_\tau.
    \end{aligned}
  \right.
\end{equation}

\begin{lemma} 
  \label{lem:love2}
  If $ \alpha-1 < \beta < 1 $ and $ y \in H^{(\alpha+1)/2}(0,T) $, then
  \begin{equation}
    \label{eq:love2}
    (y - \mathcal P_\tau y)(t_m) =
    \Dual{
      \D_{0+}^{\alpha-1-\beta}(Q_\tau - I)y',
      \D_{t_m-}^\beta \mathcal G^m
    }_{(0,t_m)}
  \end{equation}
  for each $ 1 \leqslant m \leqslant J $.
\end{lemma}
\begin{proof} 
  A straightforward calculation gives
  \begin{small}
  \begin{align*}
    & (y-\mathcal P_\tau y)(t_m) =
    (\mathcal I_\tau y-\mathcal P_\tau y)(t_m) =
    \Dual{(\mathcal I_\tau y-\mathcal P_\tau y)', 1}_{(0,t_m)} \\
    ={} &
    \Dual{
      (\mathcal I_\tau y - \mathcal P_\tau y)', \D_{t_m-}^{\alpha-1} \mathcal G^m
    }_{(0,t_m)}  \qquad\text{(by \cref{eq:def-G-wave})} \\
    ={} &
    \Dual{
      (\mathcal I_\tau y - \mathcal P_\tau y)',
      \D_{t_m-}^{\alpha-1} \mathcal G^m
    }_{(0,T)}  \qquad\text{(by the fact $ \mathcal G^m|_{(t_m,T)}=0$)} \\
    ={} &
    \Dual{
      \D_{0+}^{\alpha-1}(\mathcal I_\tau y - \mathcal P_\tau y)', \mathcal G^m
    }_{{}^0H^{(\alpha-1)/2}(0,T)} \qquad\text{(by \cref{lem:dual})} \\
    ={} & \Dual{
      \D_{0+}^{\alpha-1}(\mathcal I_\tau y - y)', \mathcal G^m
    }_{{}^0H^{(\alpha-1)/2}(0,T)} \qquad\text{(by \cref{eq:def-Ptau}).}
  \end{align*}
  \end{small}
  For any $ \alpha-1 < \beta < 1 $,
  \begin{small}
  \begin{align*}
    & \Dual{
      \D_{0+}^{\alpha-1}(\mathcal I_\tau y - y)',
      \mathcal G^m
    }_{{}^0H^{(\alpha-1)/2}(0,T)} \\
    ={} & \Dual{
      \D_{0+}^\beta \D_{0+}^{\alpha-1-\beta}(\mathcal I_\tau y - y)',
      \mathcal G^m
    }_{{}^0H^{(\alpha-1)/2}(0,T)} \\ 
    ={} & \Dual{
      \D_{0+}^{\alpha-1-\beta}(\mathcal I_\tau y - y)',
      \D_{T-}^\beta \mathcal G^m
    }_{(0,T)} \quad\text{(by \cref{lem:dual})} \\
    ={} & \Dual{
      \D_{0+}^{\alpha-1-\beta}(\mathcal I_\tau y - y)',
      \D_{t_m-}^\beta\mathcal G^m
    }_{(0,t_m)} \quad\text{(by the fact $ \mathcal G^m|_{(t_m,T)}=0 $)} \\
    ={} & \Dual{
      \D_{0+}^{\alpha-1-\beta}(Q_\tau -I)y',
      \D_{t_m-}^\beta\mathcal G^m
    }_{(0,t_m)}.
  \end{align*}
  \end{small}
  Combining the above two equations proves \cref{eq:love2} and hence this lemma.
\end{proof}

For any
\begin{small}
\[
  y \in H^{(\alpha+1)/2}(0,T;X) := \Big\{
    \sum_{n=0}^\infty c_n \phi_n:\
    \sum_{n=0}^\infty \nm{c_n}_{H^{(\alpha+1)/2}(0,T)}^2 < \infty
  \Big\},
\]
\end{small}
define 
\begin{equation}
  \label{eq:def-Ptau-X}
  \mathcal P_\tau^X y := \sum_{n=0}^\infty
  \big( \mathcal P_\tau (y, \phi_n)_X \big) \phi_n.
\end{equation}
\begin{remark}
  By \cref{eq:def-Ptau}, \cref{eq:def-Ptau-X}, \cref{lem:coer} and
  \cref{lem:regu-basic}, we obtain
  \begin{equation}
    \nm{\mathcal P_\tau^X y}_{H^{(\alpha+1)/2}(0,T;X)}
    \leqslant C_{\alpha,T} \nm{y}_{H^{(\alpha+1)/2}(0,T;X)}
  \end{equation}
  for all $ y \in H^{(\alpha+1)/2}(0,T;X) $.
\end{remark}

\begin{lemma} 
  \label{lem:cal-Ptau}
  Assume that $ y \in H^{(\alpha+1)/2}(0,T;X) \cap C^2((0,T];X) $ satisfies
  \[
    t^{-1} \nm{y'(t)}_X + \nm{y''(t)}_X \leqslant
    t^{-r}, \quad 0 < t \leqslant T,
  \]
  where $ 0 < r < 2 $. Then
  \begin{equation}
    \label{eq:cal-Ptau-inf}
    \Nm{
      \left(y-\mathcal P_\tau^X y\right)(t_m)
    }_X \leqslant C_{\alpha,\sigma,r,T}
    J^{-\sigma(2-r)} \max_{1 \leqslant j \leqslant m}
    j^{2\sigma-\sigma r+\alpha-3}
  \end{equation}
  for each $ 1 \leqslant m \leqslant J $.
\end{lemma}
\begin{proof} 
  For each $ n \in \mathbb N $, let
  \[
    y^n(t) := (y(t), \phi_n)_X, \quad 0 \leqslant t \leqslant T.
  \]
  A straightforward calculation gives
  \begin{small}
  \begin{align*}
    & \Nm{(y-\mathcal P_\tau^X y)(t_m)}_X =
    \left(
      \sum_{n=0}^\infty \Snm{
        (y^n - \mathcal P_\tau y^n)(t_m)
      }^2
    \right)^{1/2} \quad\text{(by \cref{eq:def-Ptau-X})} \\
    ={} & \left(
      \sum_{n=0}^\infty \Snm{
        \Dual{
          \D_{0+}^{\alpha-1-\beta}(I-Q_\tau)(y^n)',
          \D_{t_m-}^\beta \mathcal G^m
        }_{(0,t_m)}
      }^2
    \right)^{1/2} \quad\text{(by \cref{lem:love2})} \\
    \leqslant{} & \int_0^{t_m} \!
    \Big( \!
      \sum_{n=0}^\infty \Snm{\D_{0+}^{\alpha-1-\beta}
      (I-Q_\tau)(y^n)'}^2 \snm{\D_{t_m-}^\beta \mathcal G^m}^2 \!
    \Big)^{1/2}  \mathrm{d}t \,
    \text{(by the Minkowski inequality)} \\
    ={} & \int_0^{t_m}
    \Nm{\D_{0+}^{\alpha-1-\beta}(I-Q_\tau)y'}_X
    \snm{\D_{t_m-}^\beta \mathcal G^m} \, \mathrm{d}t,
  \end{align*}
  \end{small}
  for any $ \alpha-1 < \beta < 1 $. From \cref{lem:G2} it follows that
  \begin{small}
  \begin{align*}
    \Nm{\left(y-\mathcal P_\tau^X y\right)(t_m)}_X
    \leqslant{} & \sum_{j=1}^m (J/j)^{\sigma(\alpha-1)}
    \nm{\D_{t_m-}^\beta \mathcal G^m}_{L^1(t_{j-1},t_j)} \\
    & \qquad \times \max_{1\leqslant j\leqslant m}
    (J/j)^{\sigma(1-\alpha)} \nm{\D_{0+}^{\alpha-1-\beta}
    (I-Q_\tau)y' }_{L^\infty(t_{j-1},t_j;X)} \\
    \leqslant{} & C_{\alpha,\sigma,T}
    \max_{1 \leqslant j \leqslant m}
    (J/j)^{\sigma(1-\alpha)} \nm{
      \D_{0+}^{\alpha-1-\beta}(I-Q_\tau)y'
    }_{L^\infty(t_{j-1},t_j;X)}.
  \end{align*}
  \end{small}
  Passing to the limit $ \beta \to {1-} $ then yields
  \begin{small}
  \begin{align*}
    & \Nm{
      \left(y-\mathcal P_\tau^X y\right)(t_m)
    }_X  \leqslant C_{\alpha,\sigma,T}
    \max_{1 \leqslant j \leqslant m}
    \big(J/j\big)^{\sigma(1-\alpha)}
    \nm{
      \D_{0+}^{\alpha-2}(I-Q_\tau) y'
    }_{L^\infty(t_{j-1},t_j;X)},
  \end{align*}
  \end{small}
  so that a straightforward calculation proves \cref{eq:cal-Ptau-inf} by
  \cref{lem:I-Qtau-wave}. This completes the proof.
\end{proof}

\begin{lemma} 
  \label{lem:cal-Ptau-H1}
  Assume that $ y \in H^{(\alpha+1)/2}(0,T;X) \cap C^2((0,T];D(A^{1/2})) $
  satisfies
  \[
    t^{-1} \nm{y'(t)}_{D(A^{1/2})} +
    \nm{y''(t)}_{D(A^{1/2})} \leqslant
    t^{-r}, \quad 0 < t \leqslant T,
  \]
  where $ 0 < r < 2 $. If $ \sigma > (3-\alpha)/(2-r) $, then
  \begin{equation}
    \label{eq:cal-Ptau-H1}
    \nm{(I-\mathcal P_\tau)y}_{L^{2/\alpha}(0,T;D(A^{1/2}))}
    \leqslant C_{\alpha,\sigma,r,T} J^{\alpha-3}.
  \end{equation}
\end{lemma}
\begin{proof} 
  A simple modification of the proof of \cref{eq:cal-Ptau-inf} yields
  \begin{equation}
    \max_{1 \leqslant m \leqslant J }
    \nm{(y-\mathcal P_\tau y)(t_m)}_{D(A^{1/2})}
    \leqslant C_{\alpha,\sigma,r,T} J^{\alpha-3},
  \end{equation}
  which implies
  \[
    \nm{
      (\mathcal I_\tau - \mathcal P_\tau)y
    }_{L^\infty(0,T;D(A^{1/2}))}
    \leqslant C_{\alpha,\sigma,r,T} J^{\alpha-3}.
  \]
  It follows that
  \[
    \nm{
      (\mathcal I_\tau - \mathcal P_\tau)y
    }_{L^{2/\alpha}(0,T;D(A^{1/2}))}
    \leqslant C_{\alpha,\sigma,r,T} J^{\alpha-3}.
  \]
  In addition, a routine calculation gives
  \[
    \nm{(I-\mathcal I_\tau)y}_{
      L^{2/\alpha}(0,T;D(A^{1/2}))
    } \leqslant C_{\alpha,\sigma,r,T} J^{-2}.
  \]
  Combining the above two estimates proves \cref{eq:cal-Ptau-H1} and hence this
  lemma.
\end{proof}

\begin{lemma} 
  \label{lem:I-Pi-wave}
  If $ y \in H^{(\alpha+1)/2}(0,T) $, then
  \begin{equation}
    \label{eq:I-Pi-wave}
    \Snm{
      \left( y-\Xi_\tau^\lambda y \right)(t_m)
    } \leqslant C_{\alpha,T} \left(
      \snm{(y-\mathcal P_\tau y)(t_m)} +
      \lambda^{1/2} \nm{(I-\mathcal P_\tau)y}_{L^{2/\alpha}(0,t_m)}
    \right)
  \end{equation}
  for each $ 1 \leqslant m \leqslant J $. 
\end{lemma}
\begin{proof} 
  Letting $ \theta := (\Xi_\tau^\lambda - \mathcal P_\tau) y $, by
  \cref{eq:def-Ptau}, \cref{eq:def-Pi-wave,lem:dual} we obtain
  \[
    \Dual{
      \D_{0+}^{\alpha-1}\theta', \theta'
    }_{(0,t_m)} + \lambda \Dual{\theta, \theta'}_{(0,t_m)} =
    \lambda \Dual{y-\mathcal P_\tau y, \theta'}_{(0,t_m)},
  \]
  so that using \cref{lem:coer,lem:regu-basic} and integration by parts yields
  \begin{align*}
    \nm{\theta'}_{{}_0H^{(\alpha-1)/2}(0,t_m)}^2 +
    \lambda \snm{\theta(t_m)}^2 \leqslant
    C_\alpha \lambda \nm{(I-\mathcal P_\tau)y}_{L^{2/\alpha}(0,t_m)}
    \nm{\theta'}_{L^{2/(2-\alpha)}(0,t_m)}.
  \end{align*}
  Since
  \[
    \nm{\theta'}_{L^{2/(2-\alpha)}(0,t_m)}
    \leqslant C_{\alpha,T} \nm{\theta'}_{{}_0H^{(\alpha-1)/2}(0,t_m)},
  \]
  it follows that
  \begin{equation}
    \label{eq:theta}
    \snm{\theta(t_m)} \leqslant C_{\alpha,T}
    \lambda^{1/2} \nm{(I-\mathcal P_\tau)y}_{L^{2/\alpha}(0,t_m)}.
  \end{equation}
  Hence, \cref{eq:I-Pi-wave} follows from the triangle inequality
  \[
    \Snm{
      \left( y-\Xi_\tau^\lambda y \right)(t_m)
    } \leqslant
    \Snm{\theta(t_m)} + \Snm{(y-\mathcal P_\tau y)(t_m)}.
  \]
  This completes the proof.
\end{proof}

\medskip\noindent {\bf Proof of \cref{thm:conv-wave}.} For each $ n \in \mathbb
N $, let
\[
  u^n(t) := (u(t), \phi_n)_X, \quad 0 < t \leqslant T.
\]
By \cref{eq:U-wave}, \cref{eq:u-weak-wave}, \cref{eq:def-Pi-wave} and
\cref{lem:dual}, we have
\[
  U = \sum_{n=0}^\infty (\Xi_\tau^{\lambda_n} u^n) \phi_n,
\]
so that
\begin{small}
\begin{align*}
  & \nm{(u-U)(t_m)}_X =
  \left(
    \sum_{n=0}^\infty \snm{(u^n - \Xi_\tau^{\lambda_n} u^n)(t_m)}^2
  \right)^{1/2} \\
  \leqslant{} & C_{\alpha,T} \Big(
    \nm{(u-\mathcal P_\tau^X u)(t_m)}_X + \Big(
      \sum_{n=0}^\infty \lambda_n
      \nm{(I-\mathcal P_\tau) u^n)}_{
        L^{2/\alpha}(0,t_m)
      }^2
    \Big)^{1/2}
  \Big) \quad\text{(by \cref{eq:I-Pi-wave}).}
\end{align*}
\end{small}
Applying the Minkowski inequality gives
\[
  \Big(
    \sum_{n=0}^\infty \lambda_n
    \nm{(I-\mathcal P_\tau)u^n)}_{L^{2/\alpha}(0,t_m)}^2
  \Big)^{1/2}
  \leqslant \nm{(I-\mathcal P_\tau^X)u}_{
    L^{2/\alpha}(0,t_m;D(A^{1/2}))
  }.
\]
The above two estimates yield
\begin{small}
\[
  \nm{(u-U)(t_m)}_X \leqslant
  C_{\alpha,T} \left(
    \nm{(u-\mathcal P_\tau^X u)(t_m)}_X +
    \nm{(I-\mathcal P_\tau^X)u}_{
      L^{2/\alpha}(0,t_m;D(A^{1/2}))
    }
  \right).
\]
\end{small}
In addition, using \cref{eq:u'-l2-growth}, \cref{eq:sigma-cond-wave} and
\cref{lem:cal-Ptau} gives
\[
  \nm{(u-\mathcal P_\tau^X u)(t_m)}_X \leqslant
  C_{\alpha,\sigma,\nu,T} J^{\alpha-3}
  \nm{u_0}_{D(A^\nu)},
\]
and using \cref{eq:u'-h1-growth}, \cref{eq:sigma-cond-wave} and
\cref{lem:cal-Ptau-H1} shows
\[
  \nm{(I-\mathcal P_\tau^X)u}_{L^{2/\alpha}(0,T;D(A^{1/2}))}
  \leqslant C_{\alpha,\sigma,\nu,T} J^{\alpha-3}
  \nm{u_0}_{D(A^\nu)}.
\]
Finally, combining the above three estimates proves \cref{eq:conv-wave-inf} and
thus concludes the proof.
\hfill\ensuremath{\blacksquare}

\section{Numerical experiments}
\label{sec:numer}
This section performs two numerical experiments to verify
\cref{thm:conv-diffu,thm:conv-wave}, respectively, in the following settings: 
\begin{small}
$$\left\{\begin{array}{l}
T= 1 ;\\
  X := \big\{
    w \in H_0^1(0,1):\ w \text{ is linear on }
    \big( (m-1)/2^{11}, m/2^{11} \big) \,\,
    \text{for all $ 1 \leqslant m \leqslant 2^{11} $}
  \big\};
\\
 A: X \to X \text{ is defined by that, for any }  v \in X ,
  \int_0^1 (Av) w  = -\int_0^1 v' w'
  \quad\forall w \in X.
\end{array}
\right.
$$
\end{small}

\medskip\noindent{\bf Experiment 1.} The purpose of this experiment is to verify
\cref{thm:conv-diffu}. Let $ u_0 $ be the $ L^2 $-orthogonal projection of $
x^{0.51}(1-x) $, $ 0 < x < 1 $, onto $ X $. Define
\[
  \mathcal E_1 := \nm{U^*-U}_{L^\infty(0,T;L^2(0,1))},
\]
where $ U^* $ is the numerical solution of discretization \cref{eq:U-diffu} with
$ J = 2^{15} $ and $ \sigma = 2/\alpha $. Clearly, regarding $ \nu $ as $ 0.5 $
is reasonable. The numerical results in \cref{tab:0.2,tab:0.5,tab:0.8}
illustrate that $ \mathcal E_1 $ is close to $ O(J^{-\min\{\sigma \alpha/2,1\} }) $, which agrees well with the  estimate \cref{eq:conv-diffu-inf} in \cref{thm:conv-diffu}.

\begin{table}[H]
  \footnotesize \setlength{\tabcolsep}{2pt}
  \caption{$ \alpha=0.2 $}
  \label{tab:0.2}
  \begin{tabular}{ccccccc}
    \toprule &
    \multicolumn{2}{c}{$\sigma=1$} &
    \multicolumn{2}{c}{$\sigma=5$} &
    \multicolumn{2}{c}{$\sigma=10$} \\
    \cmidrule(r){2-3} \cmidrule(r){4-5} \cmidrule(r){6-7}
    $J$ & $\mathcal E_1$ & Order & $\mathcal E_1$ & Order & $\mathcal E_1$ & Order \\
    $2^{9}$  & 2.12e-1 & --   & 1.60e-2 & --   & 6.58e-4 & --   \\
    $2^{10}$ & 2.05e-1 & 0.05 & 1.09e-2 & 0.55 & 3.23e-4 & 1.03 \\
    $2^{11}$ & 1.97e-1 & 0.06 & 7.54e-3 & 0.54 & 1.58e-4 & 1.03 \\
    $2^{12}$ & 1.89e-1 & 0.06 & 5.23e-3 & 0.53 & 7.64e-5 & 1.05 \\
    \bottomrule
  \end{tabular}
\end{table}
\begin{table}[H]
  \footnotesize \setlength{\tabcolsep}{2pt}
  \caption{$ \alpha=0.5 $}
  \label{tab:0.5}
  \begin{tabular}{ccccccc}
    \toprule &
    \multicolumn{2}{c}{$\sigma=1$} &
    \multicolumn{2}{c}{$\sigma=2$} &
    \multicolumn{2}{c}{$\sigma=4$} \\
    \cmidrule(r){2-3} \cmidrule(r){4-5} \cmidrule(r){6-7}
    $J$ & $\mathcal E_1$ & Order & $\mathcal E_1$ & Order & $\mathcal E_1$ & Order \\
    $2^{7}$  & 1.36e-1 & -- & 3.42e-2 & -- & 3.02e-3 & -- \\
    $2^{8}$  & 1.14e-1 & 0.25 & 2.29e-2 & 0.58 & 1.45e-3 & 1.06 \\
    $2^{9}$  & 9.47e-2 & 0.27 & 1.55e-2 & 0.56 & 7.04e-4 & 1.04 \\
    $2^{10}$ & 7.76e-2 & 0.29 & 1.06e-2 & 0.55 & 3.43e-4 & 1.04 \\
    \bottomrule
  \end{tabular}
\end{table}

\begin{table}[H]
  \footnotesize \setlength{\tabcolsep}{2pt}
  \caption{$ \alpha=0.8 $}
  \label{tab:0.8}
  \begin{tabular}{ccccccc}
    \toprule &
    \multicolumn{2}{c}{$\sigma=1$} &
    \multicolumn{2}{c}{$\sigma=2$} &
    \multicolumn{2}{c}{$\sigma=2.5$} \\
    \cmidrule(r){2-3} \cmidrule(r){4-5} \cmidrule(r){6-7}
    $J$ & $\mathcal E_1$ & Order & $\mathcal E_1$ & Order & $\mathcal E_1$ & Order \\
    $2^{7}$  & 6.20e-2 & --   & 6.95e-3 & --   & 3.77e-3 & --   \\
    $2^{8}$  & 4.46e-2 & 0.48 & 3.89e-3 & 0.84 & 1.82e-3 & 1.05 \\
    $2^{9}$  & 3.22e-2 & 0.47 & 2.19e-3 & 0.83 & 8.81e-4 & 1.05 \\
    $2^{10}$ & 2.34e-2 & 0.46 & 1.24e-3 & 0.82 & 4.26e-4 & 1.05 \\
    \bottomrule
  \end{tabular}
\end{table}

\medskip\noindent{\bf Experiment 2.} The purpose of this experiment is to verify
\cref{thm:conv-wave}. Let $ u_0 $ be the $ L^2 $-orthogonal projection of $
x^{1.51}(1-x)^2 $, $ 0 < x < 1 $, onto $ X $. Let
\[
  \mathcal E_2 := \max_{1 \leqslant j \leqslant j}
  \nm{(U^*-U)(t_j)}_{L^2(\Omega)},
\]
where $ U^* $ is the numerical solution of discretization \cref{eq:U-wave} with
$ J=2^{15} $ and $ \sigma=2(3-\alpha)/\alpha $. Evidently, regarding $ u_0 \in
D(A) $ is reasonable. The numerical results in \cref{tab:wave} clearly
demonstrate that $ \mathcal E_2 $ is close to $ O(J^{\alpha-3}) $, which
agrees well with \cref{thm:conv-wave}.

\begin{table}[H]
  \footnotesize \setlength{\tabcolsep}{2pt}
  \caption{$ \sigma=2(3-\alpha)/\alpha $}
  \label{tab:wave}
  \begin{tabular}{ccccccc}
    \toprule &
    \multicolumn{2}{c}{$\alpha=1.2$} &
    \multicolumn{2}{c}{$\alpha=1.5$} &
    \multicolumn{2}{c}{$\alpha=1.8$} \\
    \cmidrule(r){2-3} \cmidrule(r){4-5} \cmidrule(r){6-7}
    $J$ & $\mathcal E_2$ & Order & $\mathcal E_2$ & Order & $\mathcal E_2$ & Order \\
    $2^{6}$ & 2.44e-5 & --   & 1.27e-4 & --   & 7.97e-4 & --   \\
    $2^{7}$ & 6.81e-6 & 1.84 & 4.58e-5 & 1.47 & 3.57e-4 & 1.16 \\
    $2^{8}$ & 1.90e-6 & 1.84 & 1.65e-5 & 1.47 & 1.57e-4 & 1.18 \\
    $2^{9}$ & 5.35e-7 & 1.83 & 5.97e-6 & 1.47 & 6.87e-5 & 1.20 \\
    \bottomrule
  \end{tabular}
\end{table}

\section{Conclusions} 
For the fractional evolution equation, we have analyzed a low-order discontinuous
Galerkin (DG) discretization with fractional order $ 0 < \alpha < 1 $ and a
low-order Petrov Galerkin (PG) discretization with fractional order $ 1 < \alpha
< 2 $.  When using   uniform temporal grids, the two discretizations are equivalent to
the L1 scheme with $ 0 < \alpha < 1 $ and the L1 scheme with $ 1 < \alpha < 2 $,
respectively.  For the DG discretization with graded temporal grids,  sharp
error estimates are rigorously established for smooth and nonsmooth initial data.
For the PG discretization, the optimal $ (3-\alpha) $-order temporal accuracy is
derived on appropriately graded temporal grids. The theoretical results have been verified by   numerical results.

However, our analysis of the PG discretization requires $ u_0 \in D(A^\nu) $
with $ 1/2 < \nu \leqslant 1 $. Hence, how to analyze the case $ 0 < \nu
\leqslant 1/2 $ remains an open problem. It appears that the results and
techniques developed in this paper can be used to analyze the semilinear
fractional diffusion-wave equations with graded temporal grids, and this is our
ongoing work.

\section*{Acknowledgements}

Binjie Li was supported in part by the National Natural Science Foundation of China (NSFC) Grant No. 11901410, Xiaoping Xie was supported in part by the National Natural Science Foundation of China (NSFC) Grant No. 11771312, and Tao Wang was supported in part by the China Postdoctoral Science Foundation (CPSF) Grant No. 2019M66294.

\clearpage

\appendix
\section{Properties of fractional calculus operators}
\begin{lemma}
  \label{lem:coer}
  For any $ v \in {}_0H^\gamma(a,b) $ with $ 0 < \gamma < 1/2 $,
  \begin{align*}
    \cos(\gamma\pi) \nm{\D_{a+}^\gamma v}_{L^2(a,b)}^2 \leqslant
    \Dual{ \D_{a+}^\gamma v, \D_{b-}^\gamma v }_{(a,b)} \leqslant
    \sec(\gamma\pi) \nm{\D_{a+}^\gamma v}_{L^2(a,b)}^2, \\
    \cos(\gamma\pi) \nm{\D_{b-}^\gamma v}_{L^2(a,b)}^2 \leqslant
    \Dual{ \D_{a+}^\gamma v, \D_{b-}^\gamma v}_{(a,b)} \leqslant
    \sec(\gamma\pi) \nm{\D_{b-}^\gamma v}_{L^2(a,b)}^2.
  \end{align*}
\end{lemma}

\begin{lemma}
  \label{lem:regu-basic}
  For any $ v \in {}_0H^\gamma(a,b) $ and $ w \in {}^0H^\gamma(a,b) $ with $ 0 <
  \gamma < \infty $,
  \begin{small}
  \begin{align*}
    C_1 \nm{v}_{{}_0H^\gamma(a,b)} \leqslant
    \nm{\D_{a+}^\gamma v}_{L^2(a,b)} \leqslant
    C_2 \nm{v}_{{}_0H^\gamma(a,b)}, \\
    C_1 \nm{w}_{{}^0H^\gamma(a,b)} \leqslant
    \nm{\D_{b-}^\gamma w}_{L^2(a,b)} \leqslant
    C_2 \nm{w}_{{}^0H^\gamma(a,b)},
  \end{align*}
  \end{small}
  where $ C_1 $ and $ C_2 $ are two positive constants depending only on $
  \gamma $.
\end{lemma}

\begin{lemma}
  \label{lem:dual}
  Assume that $ v \in {}_0H^{\gamma/2}(a,b) $ and $ w \in {}^0H^{\gamma/2}(a,b)
  $ with $ 0 < \gamma < 1 $. Then
  \begin{equation}
    \Dual{\D_{a+}^\gamma v, w}_{{}^0H^{\gamma/2}(a,b)} =
    \Dual{\D_{b-}^\gamma w, v}_{{}_0H^{\gamma/2}(a,b)}.
  \end{equation}
  If $ \D_{a+}^\gamma v \in L^{2/(1+\gamma)}(a,b) $, then
  \begin{equation}
    \Dual{\D_{a+}^\gamma v, w}_{{}^0H^{\gamma/2}(a,b)} =
    \Dual{\D_{a+}^\gamma v, w}_{(a,b)}.
  \end{equation}
  If $ \D_{b-}^\gamma w \in L^{2/(1+\gamma)}(a,b) $, then
  \begin{equation}
    \Dual{\D_{b-}^\gamma w, v}_{{}_0H^{\gamma/2}(a,b)} =
    \Dual{\D_{b-}^\gamma w, v}_{(a,b)}.
  \end{equation}
\end{lemma}

\noindent For the proof of \cref{lem:coer}, we refer the reader to
\cite{Ervin2006}. For the proof of \cref{lem:regu-basic}, we refer the reader to
\cite{Luo2019}. Since the proof of \cref{lem:dual} is a standard density
argument by \cref{lem:coer,lem:regu-basic}, it is omitted here.


\section{ Some inequalities}
\begin{lemma} 
  \label{lem:31}
  For any $ 0 < \beta < 1 $ and $ 0 \leqslant t < a < b < c < d $,
  \begin{equation} \label{eq:lem31}
    \frac{
      (d-t)^{1-\beta} - (d-a)^{1-\beta}
    }{
      (d-a)^{1-\beta} - (d-b)^{1-\beta}
    } >
    \frac{
      (c-t)^{1-\beta} - (c-a)^{1-\beta}
    }{
      (c-a)^{1-\beta} - (c-b)^{1-\beta}
    }.
  \end{equation}
\end{lemma}
\begin{proof} 
  Let
  \[
    w(y) := \begin{cases}
      \beta/(1-\beta) & \text{ if } y = 1, \\
      \frac{1-y^{-\beta}}{y^{1-\beta} - 1} &
      \text{ if } y \in [0,\infty) \setminus \{1\}.
    \end{cases}
  \]
  A routine argument proves that $ w $ is strictly decreasing on $ [0,\infty) $,
  so that
  \[
    w\big( (d-t-x)/(d-a-x) \big) < w\big( (d-b-x)/(d-a-x) \big)
    \quad \forall 0 \leqslant x \leqslant d-c.
  \]
  It follows that, for any $ 0 \leqslant x \leqslant d-c $,
  \begin{small}
  \[
    \frac{
      (d-a-x)^{-\beta} -  (d-t-x)^{-\beta}
    }{
      (d-t-x)^{1-\beta} - (d-a-x)^{1-\beta}
    } < \frac{
      (d-b-x)^{-\beta} - (d-a-x)^{-\beta}
    }{
      (d-a-x)^{1-\beta} - (d-b-x)^{1-\beta}
    },
  \]
  \end{small}
  which implies
  \begin{small}
  \begin{align*}
    & \big( (d-a-x)^{-\beta} - (d-t-x)^{-\beta} \big)
    \big( (d-a-x)^{1-\beta} - (d-b-x)^{1-\beta} \big) \\
    & -
    \big( (d-b-x)^{-\beta} - (d-a-x)^{-\beta} \big)
    \big( (d-t-x)^{1-\beta} - (d-a-x)^{1-\beta} \big) < 0
  \end{align*}
  \end{small}
  for all $ 0 \leqslant x \leqslant d-c $. A simple calculation then yields $
  g'(x) < 0 $ for all $ 0 \leqslant x \leqslant d-c $, where
  \[
    g(x) := \frac{
      (d-t-x)^{1-\beta} - (d-a-x)^{1-\beta}
    }{
      (d-a-x)^{1-\beta} - (d-b-x)^{1-\beta}
    }, \quad 0 \leqslant x \leqslant d-c.
  \]
  This proves $ g(d-c) < g(0) $, namely \cref{eq:lem31}, and thus concludes the
  proof.
\end{proof}

\begin{lemma}
  \label{lem:pre-G}
  For any $ 0 < \beta < 1 $, $ \mu \geqslant 0 $ and $ 0 \leqslant t < a < b < c
  < d $,
  \begin{small}
  \begin{equation}
    \label{eq:pre-G}
    \frac{
      (d-t)^{1-\beta} - (d-a)^{1-\beta} + \mu (a-t)
    }{
      (d-a)^{1-\beta} - (d-b)^{1-\beta} + \mu (b-a)
    } >
    \frac{
      (c-t)^{1-\beta} - (c-a)^{1-\beta} + \mu (a-t)
    }{
      (c-a)^{1-\beta} - (c-b)^{1-\beta} + \mu (b-a)
    }.
  \end{equation}
  \end{small}
\end{lemma}
\begin{proof}
  Define
  \[
    g(s):= \big( d-b+s(b-a) \big)^{1-\beta}
    - \big( c-b+s(b-a) \big)^{1-\beta}
    \quad \forall 0 \leqslant s \leqslant a.
  \]
  By the mean value theorem, there exists $ \theta \in (0,1)$ such that
  \begin{align*}
    & g(1)-g(0)= g'(\theta) \\
    ={} &
    (1-\beta) \left(
      \big( d-b+\theta(b-a) \big)^{-\beta}
      -\big( c-b+\theta(b-a) \big)^{-\beta}
    \right) (b-a).
  \end{align*}
  Since
  \begin{small}
  \[
    \big( d-b+\theta(b-a) \big)^{-\beta} -
    \big( c-b+\theta(b-a) \big)^{-\beta} >
    (d-a)^{-\beta}-(c-a)^{-\beta},
  \]
  \end{small}
  it follows that
  \[
    g(1)-g(0) > (1-\beta) ((d-a)^{-\beta}-(c-a)^{-\beta}) (b-a),
  \]
  which implies
  \begin{small}
  \begin{align}
    \frac{1}{b -a} > \frac{
      (1-\beta)((d-a)^{-\beta}-(c-a)^{-\beta})
    }{
      (d-a)^{1-\beta} - (d-b)^{1-\beta}- (c-a)^{1-\beta} + (c-b)^{1-\beta}
    }.
  \end{align}
  \end{small}
  Hence, by the estimate
  \[
    (d-s)^{-\beta} - (c-s)^{-\beta} <
    (d-a)^{-\beta} - (c-a)^{-\beta}
    \quad \forall 0 \leqslant s < a,
  \]
  we obtain
  \begin{small}
  \begin{align}
    \frac{1}{b -a} > \frac{
      (1-\beta)((d-s)^{-\beta}-(c-s)^{-\beta})
    }{
      (d-a)^{1-\beta} - (d-b)^{1-\beta}- (c-a)^{1-\beta} + (c-b)^{1-\beta}
    } \quad \forall 0 \leqslant s < a.
  \end{align}
  \end{small}
  Integrating both sides of the above equation with respect to $ s $ from $ t $
  to $ a $ yields
  \begin{small}
  \begin{equation}
    \label{eq:adam}
    \frac{a -t}{b -a} >
    \frac{
      (c-t)^{1-\beta}- (c-a)^{1-\beta} -
      (d-t)^{1-\beta} + (d-a)^{1-\beta}
    }{
      (c-a)^{1-\beta} - (c-b)^{1-\beta} -
      (d-a)^{1-\beta} + (d-b)^{1-\beta}
    }.
  \end{equation}
  \end{small}
  Let
  \begin{align*}
    & \mathcal A := (d-t)^{1-\beta} - (d-a)^{1-\beta}, \quad
    \mathcal B := (d-a)^{1-\beta} - (d-b)^{1-\beta}, \\
    & \mathcal C := (c-t)^{1-\beta} - (c-a)^{1-\beta}, \quad
    \mathcal D := (c-a)^{1-\beta} - (c-b)^{1-\beta}, \\
    & \mathcal M := \mu (a-t), \qquad \mathcal N := \mu (b-a).
  \end{align*}
  Since \cref{lem:31} implies $ \mathcal A \mathcal D > \mathcal B \mathcal C $
  and \cref{eq:adam} implies $ \mathcal M (\mathcal D-\mathcal B) \geqslant
  \mathcal N (\mathcal C - \mathcal A) $, we obtain
  \[
    (\mathcal A + \mathcal M)(\mathcal D + \mathcal N) >
    (\mathcal B + \mathcal N)(\mathcal C + \mathcal M),
  \]
  which proves \cref{eq:pre-G}. This completes the proof.
\end{proof}

\begin{lemma} 
  \label{lem:G-wave}
  For any $ 1 < \beta < 2 $ and $ 0 \leqslant t < a < b \leqslant c $,
  \begin{equation}
    \label{eq:lxy}
    \frac{
      (c-t)^{2-\beta} - (c-a)^{2-\beta}
    }{
      (c-a)^{2-\beta} - (c-b)^{2-\beta}
    } < \frac{a-t}{b-a}.
  \end{equation}
\end{lemma}
\begin{proof} 
  By the mean value theorem, there exists $ 0 < \theta < 1 $ such that
  \[
    (c-a)^{2-\beta} - (c-b)^{2-\beta} =
    (2-\beta) \big( c-b+\theta(b-a) \big)^{1-\beta}
    \, (b-a),
  \]
  and so
  \[
    \frac{(2-\beta)(c-a)^{1-\beta}}{
      (c-a)^{2-\beta} - (c-b)^{2-\beta}
    } = \left( \frac{c-a}{ c-b+\theta(b-a) } \right)^{1-\beta}
    \frac1{b-a} < \frac1{b-a}.
  \]
  Since
  \[
    (c-a)^{1-\beta} \geqslant (c-s)^{1-\beta}
    \quad\text{for all } 0 \leqslant s \leqslant a,
  \]
  it follows that
  \[
    \frac{(2-\beta)(c-s)^{1-\beta}}{
      (c-a)^{2-\beta} - (c-b)^{2-\beta}
    } < \frac1{b-a} \quad\text{for all }
    0 \leqslant s \leqslant a.
  \]
  Hence, for any $ 0 \leqslant t < a $,
  \begin{align*}
    \int_t^a \frac{(2-\beta)(c-s)^{1-\beta}}{
      (c-a)^{2-\beta} - (c-b)^{2-\beta}
    } \, \mathrm{d}t <
    \int_t^a \frac1{b-a} \, \mathrm{d}t,
  \end{align*}
  which implies \cref{eq:lxy}. This completes the proof.
\end{proof}

\begin{lemma} 
  \label{lem:discr-conv1}
  If $ \beta>-1 $ and $ \gamma>1 $, then
  \begin{equation}
    \sum_{j=1}^{k-1} j^\beta (k^\sigma-j^\sigma)^{-\gamma}
    \leqslant C_{\beta,\gamma,\sigma} k^{\beta-(\sigma-1)\gamma}
  \end{equation}
  for all $ k \geqslant 2 $.
\end{lemma}
\begin{proof}
  A routine calculation gives
  \begin{align*}
    C_0 \leqslant \frac{
      j^\beta(k^\sigma-j^\sigma)^{-\gamma}
    }{
      (j-x)^\beta(k^\sigma-(j-x)^\sigma)^{-\gamma}
    } \leqslant C_1,
  \end{align*}
  for all $ 2 \leqslant j \leqslant k-1 $ and $ 0 < x \leqslant 1 $, where $ C_0
  $ and $ C_1 $ are two positive constants depending only on $ \beta $, $ \gamma
  $ and $ \sigma $. Hence,
  \begin{align*}
    & \sum_{j=1}^{k-1} j^\beta(k^\sigma-j^\sigma)^{-\gamma} \\
    \leqslant{} &
    C_{\beta,\gamma,\sigma} \int_1^{k-1}
    x^\beta(k^\sigma - x^\sigma)^{-\gamma} \, \mathrm{d}x \\
    \leqslant{} &
    C_{\beta,\gamma,\sigma} k^{-\sigma\gamma+\beta+1}
    \int_{k^{-\sigma}}^{((k-1)/k)^\sigma}
    s^{(1+\beta)/\sigma-1}(1-s)^{-\gamma} \, \mathrm{d}s \\
    \leqslant{} &
    C_{\beta,\gamma,\sigma} k^{-\sigma\gamma+\beta+1}
    \int_0^{((k-1)/k)^\sigma}
    s^{(1+\beta)/\sigma-1}(1-s)^{-\gamma} \, \mathrm{d}s \\
    \leqslant{} &
    C_{\beta,\gamma,\sigma} k^{-\sigma\gamma+\beta+1}
    \big( 1-((k-1)/k)^\sigma \big)^{1-\gamma} \\
    \leqslant{} &
    C_{\beta,\gamma,\sigma} k^{-\sigma\gamma+\beta+1+\gamma-1} \\
    ={} &  C_{\beta,\gamma,\sigma}
    k^{\beta - (\sigma-1)\gamma}.
  \end{align*}
  This proves the lemma.
\end{proof}

A trivial modification of the proof of \cref{lem:discr-conv1} yields the
following estimate.
\begin{lemma} 
  \label{lem:discr-conv2}
  If $ \beta > -1 $ and $ 1/2 \leqslant \gamma < 1 $, then
  \begin{equation}
    \sum_{j=1}^{k-1} j^\beta
    (k^\sigma-j^\sigma)^{-\gamma}
    \leqslant C_{\beta,\sigma} (1-\gamma)^{-1}
    k^{\beta-\sigma\gamma+1}
  \end{equation}
  for all $ k \geqslant 2 $.
\end{lemma}

\end{document}